\documentclass[11pt]{article}
\usepackage{amssymb,amsmath,setspace,graphicx,multicol,wrapfig,amsfonts,amsthm,titling,url,array,parskip,enumitem,bbm,color,tikz,tikz-cd,caption,mathdots}
\usepackage{mathtools}
\usepackage{placeins,comment}
\usepackage[colorlinks=true, pdfstartview=FitV, linkcolor=blue, citecolor=blue, urlcolor=blue]{hyperref}
\usepackage[capitalize,nameinlink,noabbrev,nosort]{cleveref}
\usepackage{aurical,pbsi}
\usepackage[centertableaux]{ytableau}
\usepackage[T1]{fontenc}
\usepackage[hmargin=2.5cm,vmargin=2.5cm]{geometry}
\usepackage{todonotes}
\ytableausetup{boxsize=0.8em} 

\usepackage[backend=biber, style=alphabetic, maxbibnames=99, maxalphanames=5, url=false, isbn=false, doi=false]{biblatex}
\usepackage{amssymb}
\graphicspath{ {./figs/} }
\usepackage[all]{nowidow} 
\usepackage[protrusion=true,expansion=true]{microtype} 

\usepackage[usenames,dvipsnames]{xcolor}

\usepackage[mathscr]{euscript}
\usepackage{svg}
\usepackage{svg-extract}
\newtheorem*{rep@theorem}{\rep@title}
\newcommand{\newreptheorem}[2]{%
\newenvironment{rep#1}[1]{%
 \def\rep@title{#2 \ref{##1}}%
 \begin{rep@theorem}}%
 {\end{rep@theorem}}}

\numberwithin{equation}{section}

\usepackage{thmtools}
\theoremstyle{plain}
\newtheorem{theorem}{Theorem}[section]
\newreptheorem{theorem}{Theorem}

\newtheorem{lemma}[theorem]{Lemma}

\newtheorem{conjecture}[theorem]{Conjecture}
\theoremstyle{definition}
\newtheorem{definition}[theorem]{Definition}
\newtheorem{example}[theorem]{Example}
\newtheorem{question}[theorem]{Question}
\newtheorem{remark}[theorem]{Remark}
\crefname{definition}{Definition}{Definitions}
\crefname{theorem}{Theorem}{Theorems}
\crefname{lemma}{Lemma}{Lemmas}
\crefname{example}{Example}{Examples}

\newcommand\NN{\mathbb{N}}
\newcommand\RR{\mathbb{R}}

\hypersetup{ 	
pdfsubject = {},
pdftitle = {Parametrizing the grassmannian using pipe dreams},
pdfauthor = {Kartik Singh}
}

\title{Parametrizing the Grassmannian using pipe dreams}
\author{Kartik Singh}

\begin{document}
\maketitle
\abstract{Postnikov gave a parametrization for the totally non-negative Grassmannian using the matroid decomposition and associating a network with \reflectbox{L}-diagrams. Talaska and Williams extend this result to the entire Grassmannian by using the Deodhar decomposition instead of the matroid decomposition, and the networks this time are associated with the generalized versions of \reflectbox{L}-diagrams, which are called Go-diagrams. We provide an alternative parametrization for the Deodhar components, this time constructing a network based on the pipe dreams associated with the Go-diagrams. 

This parametrization has several nice properties; for one, it allows us to easily calculate the image of a point under the isomorphism $Gr_{k,n}\simeq Gr_{n-k,n}$. The second feature of this parametrization is that if we write the Pl\"ucker coordinates using the Lindst\"orm-Gessel-Viennot (LGV) lemma in our parametrization, we can associate a pipe dream to each summand, which allows us to reveal additional structure on the summands. Finally, as an application of our parametrization, we describe a case where we can conclude whether one Deodhar component lies inside the closure of another.}

\section{Introduction}
The \textbf{Grassmannian} $Gr_{k,n}$ is the set of $k$-dimensional spaces of $\mathbb{F}^n$ where $\mathbb{F}$ is a field. Algebraically we can see it as a coset $SL_n(\mathbb{F})/P$ where $P$ is a parabolic subgroup of $SL_n(\mathbb{F})$. To each point in the Grassmannian, we can assign Pl\"{u}cker coordinates, which can be used to give the Grassmannian a structure of a projective variety. If our field is $\mathbb{R}$, then the Grassmannian has a very interesting subset called the totally non-negative Grassmannian which is defined as the set of points which have non-negative Pl\"{u}cker coordinates. 

For a point in the Grassmannian, the set of Pl\"{u}cker coordinates which are non-zero form a basis for a matroid. This allows us to define the \textbf{matroid strata} associated with the matroid $\mathcal{M}$ as the collection of points in the Grassmannian, whose set of non-zero Pl\"{u}cker coordinates forms a basis for $\mathcal{M}$. This decomposition has a very non-trivial geometry; however, its intersection with the totally non-negative Grassmannian is much more well-behaved. Not all matroids have a non-empty intersection with the totally non-negative Grassmannian; those that do are referred to as positroids. Postnikov showed that the totally non-negative Grassmannian has a $CW$-complex structure where the cells are given by the positroid components \cite{postnikov2006total}. He further gave a parametrization for each of these positroid cells by assigning a \reflectbox{L}-diagram to each positroid. Using the \reflectbox{L}-diagram, he constructed a planar network whose associated weight matrix corresponds to a point in the positroid cell.

Talaska and Williams extended this result to obtain a parametrization for the entire Grassmannian \cite{talaskawilliams}, where they replaced the positroid cells with Deodhar components, and \reflectbox{L}-diagrams with \textbf{Go-diagrams}, which are generalizations of \reflectbox{L}-diagrams. Deodhar components have a simple topology; each component is homeomorphic to $(\mathbb{F}^*)^a\times \mathbb{F}^b$, and we can attach a Go-diagram to each component. The intersection of these with the non-negative Grassmannian is simple to understand as well; the intersection is non-empty if and only if the Go-diagram corresponding to the component is also a \reflectbox{L}-diagram. In this case, it is the same as the positroid cell corresponding to the \reflectbox{L}-diagram. Talaska and Williams extended Postnikov's result by constructing a network for each Go-diagram (which in this case may no longer be planar), such that the weight matrix of this network corresponds to a point in the Deodhar component. When the Go-diagram is a \reflectbox{L}-diagram, their construction coincides with Postnikov's network for positroid cells.

Deodhar decomposition, even though not as popular as the other decompositions of the Grassmannian, is interesting in its own right. Deodhar first introduced the decomposition in the flag variety to understand the Richardson variety. Kazhdan and Lusztig \cite{Lusztig1979} defined $R$-polynomials to give a recursive formula for the Kazhdan--Lusztig polynomials. $R$-polynomials themselves can be computed by counting the number of points in the Richardson cell in the setting when our field is finite $\mathbb{F}_q$. Deodhar decomposed each Richardson cell as a union of Deodhar components, and since each Deodhar component has a simple topology, this allowed him to compute the $R$-polynomials easily. 

Deodhar decomposition for the Grassmannian was introduced by Kodama and Williams in their work on soliton solutions of the KP equation \cite{kodamawilliams}. It is well known that to each point $A$ in the real Grassmannian, we can construct a solition solution $u_A(x,y,t)$ of the KP equation. Kodama and Williams used Deodhar components to describe the \textit{contour plots} of the solutions at $t\ll0$.

The Deodhar component, although it has a simple topology and appears to have interesting properties, the decomposition itself is poorly behaved compared to other decompositions. In particular, it is not a stratification (i.e., the closure of a component is not a union of other components) \cite{dudas2008note}. This makes the problem of understanding the intersection of the closure of one component with another component highly non-trivial. Marcott \cite{marcott} proposed a conjecture regarding when a component can lie in the closure of another. However, even though his conjecture doesn't hold in general (see \cref{rem:closing}), we were able to show a special case of the conjecture.

Both Go-diagrams and \reflectbox{L}-diagrams have a pipe dream associated with them. In this paper, we will provide an alternative parametrization for the Deodhar components by constructing a network based on these pipe dreams. We will refer to this as the \textbf{restricted path parametrization}. The motivation for introducing these concepts stems from the desire to answer the question: given a Deodhar component with a Go-diagram $D$, which Deodhar components lie on its codimension-one boundary? Since each component is homeomorphic to $(\mathbb{F}^*)^a\times \mathbb{F}^b$, one way to get something in the codimension one is to take a limit of points with one parameter with values in $\mathbb{F}^*$ that is approaching 0. On the level of Go-diagrams, this corresponds to changing filling at a fixed cell or a pair of cells. Therefore, we would like to understand how the parametrization changes when we change the filling at a fixed cell. The changes to the Talaska-Williams network are non-trivial under such changes, and switching to pipe dreams to parametrize these components fixes this problem; each point in the component can be written as a submatrix of a square matrix, which has a product formula that can be easily read off from the pipe dream.

There are several advantages of this construction. We have an isomorphism $Gr_{k,n}\simeq G_{n-k,n}$. If we have a point in $Gr_{k,n}$ along with the parameters corresponding to it, we can write the image of this point under this isomorphism by simply "dualizing" the network on the pipe dreams. This parametrization also gives a new way to look at the Pl\"{u}cker coordinates. Pl\"ucker coordinates can be understood using the Lindstr\"om-Gessel-Viennot (LGV) lemma on the Talaska--Williams network or the pipe dream network. However, the LGV lemma doesn't provide any additional structure to the summands that appear. In our restricted path parametrization, we attach a pipe dream to each of the summands, and then show that these pipe dreams can be obtained by performing a sequence of "toggling" moves on our Go-diagram, revealing additional structure on these summands. Finally, as an application of this parametrization, we describe some of the Deodhar components inside the codimension one boundary of a fixed Deodhar component:
\begin{reptheorem}{thm:closemain}
        Let $D'$ be a Go-diagram in some partition $\lambda$, and let $c\prec c'$ be a crossing and uncrossing pair in $D'$ with $D'(c) = \circ$. Suppose the pipes passing through $c$ are labelled $i$ and $i+1$ for some $1\leq i<n$.
        Let $D$ be the diagram obtained by replacing the stones at $c$ and $c'$ with $+$.
        Then $D$ is a Go-diagram and $\mathcal{D}_D' \subset \overline{\mathcal{D}_D}$.
    \end{reptheorem}

Our paper is structured as follows: \cref{sec:back} discusses the various decompositions, along with combinatorial structures associated with each of these decompositions. \cref{sec:respaths} defines the restricted path parametrization and looks at the properties of this parametrization. Finally, \cref{sec:closure} discusses the application of this parametrization in the problem of understanding the closure of Deodhar components.
\section{Background}\label{sec:back}
\subsection{Partitions, Go-diagrams and \reflectbox{L}-diagrams}
Here, we will introduce the combinatorial objects with which we will be interacting in this article. Let $n,k\in \NN, k\leq  n$. Let $\mathfrak{S}_n$ denote the symmetric group on $n$ symbols. Let $1$ denote the identity element in $\mathfrak{S}_n$, and if $v\in \mathfrak{S}_n$, then let $\ell(v)$ denote the length of $v$. For $w\in \mathfrak{S}_n$, let $\textbf{w} = s_{i_1}s_{i_2}...s_{i_m}$ be a reduced word presentation of $w$. A \textbf{subexpression} of $\textbf{w}$ is one where we replace some of the factors in $\textbf{w}$ with $1$. Let $\textbf{v}$ be a subexpression of $\textbf{w}$. For $j\leq m$, let $v_{(j)}$ denote the product of the first $j$ factors in $\textbf{v}$. Finally let $\pi(\textbf{v})= v_{(m)}$. 
\begin{definition}[Distinguished subexpressions]\cite{deodhar1, deodhar2}
    A subexpression $\textbf{v}$ of $\textbf{w}$ is said to be \textbf{distinguished} if $\ell(v_{(j+1)})\leq \ell(v_{(j)}s_{i_{j+1}})$ for all $j<m$. We use the notation $\textbf{v}\prec \textbf{w}$ to say $\textbf{v}$ is a distinguished subexpression of $\textbf{w}$. A distinguished subexpression $\textbf{v}$ of $\textbf{w}$ is said to be \textbf{positive distinguished expression} if $\ell(v_{j+1})\geq \ell(v_j)$ for all $j$.
\end{definition}
\begin{example}\label{ex:distinguished}
    Let $n=6$, and let $\textbf{w}=s_3s_4s_5s_2s_3s_4s_1s_2s_3$. Let us look at the following subexpressions of $\textup{w}$:
    \begin{align*}
        \textbf{v}^1&=s_31111111s_3\\
        \textbf{v}^2&=s_3111s_3111s_3\\
        \textbf{v}^3&=11111111s_3
    \end{align*}
    Then $\textbf{v}^1$ is a not a distinguished subexpression as $\ell(v^1_{(4)})=1\not\leq \ell (v^1_{(4)}s_3)=0$, $\textbf{v}^2$ is a distinguished subexpression but not a positive distinguished subexpression as $\ell(v^2_{(3)})>\ell(v^2_{(4)})$. Lastly, $\textbf{v}^3$ is a positive distinguished subexpression.
\end{example}
A \textbf{partition} $\lambda= (\lambda_1,\lambda_2,...,\lambda_k)$ bounded by $n,k$ is a sequence of $k$ weakly decreasing non-negative integers with the additional restriction $\lambda_1\leq n-k$. Let $\Lambda_{k,n}$ denote the set of such partitions. For $\lambda\in \Lambda_{k,n}$, define $|\lambda|:= \lambda_1+ \lambda_2+...+\lambda_k$ The Young diagram corresponding to $\lambda\in \Lambda_{k,n}$ lies inside a $k\times (n-k)$ rectangle ($k$ is the number of rows). We will use English notation to draw our Young diagrams. We shall use $\lambda$ also to denote its Young diagram. We can order the cells in $\lambda$ to form a poset. For two cells $b,c$ in $\lambda$, we say $b\prec c$ if $c$ is weakly north-west of $b$. We also use the notation $c_{in}$ to denote the set of all cells $b$ with $b\prec c$.
         
Let $\binom{[n]}{k}$ denote the subsets of $[n]$ with $k$ elements. The set $\Lambda_{k,n}$ in bijection with $\binom{[n]}{k}$. For a given $\lambda$, label the south-east border of the Young diagram with integers from $1$ to $n$ starting from the eastmost edge. The labels on the vertical edges give the subset $I_\lambda$. For partitions $\lambda, \mu$, we say $\lambda \leq \mu$ if the Young diagram of $\lambda$ is contained entirely inside the one for $\mu$. We use this to define an order on $\binom{[n]}{k}$ as: $\lambda \leq \mu \iff I_\lambda \leq I_\mu$. We also have a bijection between $\binom{[n]}{k}$ and the set of Grassmannian permutations with descent at position $k$; $I_\lambda$ is simply the set of entries which show up in the permutation after the descent. Let $w_\lambda$ be the Grassmannian permutation corresponding to $\lambda$. See \cref{ex:pipedreams}.

A \textbf{pipe dream} is a filling of $\lambda$ using elbow pipes \includesvg[width=5mm, height=5mm]{elbow} and crossings \includesvg[width=5mm, height=5mm]{cross}. The numbering of the Young diagram on the south-east border indexes the pipes. We follow the pipes to the north-west boundary, and we label the north-west border with pipes exiting at that edge. Reading the numbers from right to left gives us the permutation corresponding to the pipe dream. For a pipedream $P$, we use $\pi(P)$ to denote this permutation. See \cref{ex:pipedreams}.

Each pipe dream corresponds to a subexpression of $w_\lambda$ corresponding to a reading word, which we shall describe here. We first attach a simple transposition to every cell in the Young diagram of $\lambda$. We start by attaching $s_{n-k}$ to the top left cell of $\lambda$. Now we attach transpositions recursively: if the cell $b$ has transposition $s_i$, then we attach transpositions $s_{i-1}$ and $s_{i+1}$ to the cells right and bottom of $b$ respectively.  

Let $|\lambda| = m$. A \textbf{reading word} on $\lambda$ is a filling of $\lambda$ with numbers from $1$ to $m$, such that entries are decreasing from left to right and top to bottom. Note that a reduced word presentation $\textbf{w}_\lambda$ for $w_\lambda$ is obtained if we fill all the cells with crossings. Now, to get the subexpression corresponding to the pipe dream, we put $1$ on the entries corresponding to elbow pipes in the reading word, and the attached transposition on the entries corresponding to crossings in the reading word. See \cref{ex:pipedreams} 
\begin{figure}
            \centering
            \includesvg[width =60mm, height=30mm]{readingword}
            \caption{Attaching transpositions and an example for a reading word}
            \label{fig:readingword}
        \end{figure}
        \begin{figure}
            \centering
            \includesvg[width =90mm, height=30mm]{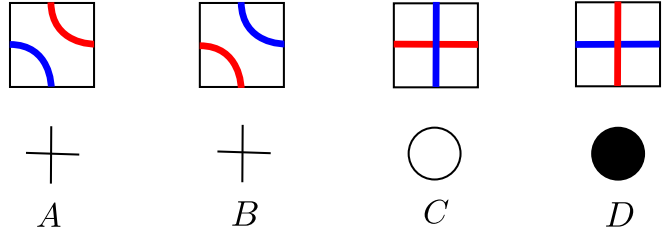}
            \caption{Possible cases at a cell. The lower index pipe is coloured in red.}
            \label{fig:configs}
        \end{figure}
        
        \begin{figure}
            \centering
            \includesvg[width =90mm, height=30mm]{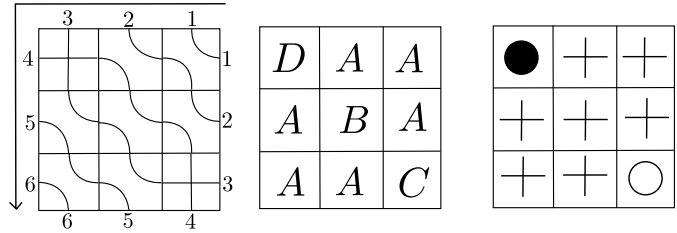}
            \caption{Example of a pipe dream which is not a Go-diagram.}
            \label{fig:pipedreamex2}
        \end{figure}
        \begin{figure}
            \centering
            \includesvg[width =90mm, height=30mm]{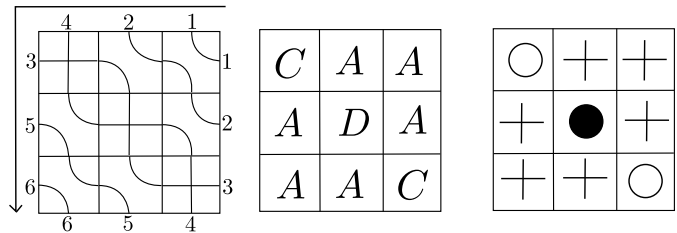}
            \caption{Example of a pipe dream, which is a Go-diagram.}
            \label{fig:pipedreamex1}
        \end{figure}

At any cell, we have two pipes entering it; we shall compare the indices of the two pipes. We have four possible cases, which are illustrated in the \cref{fig:configs}; here, the lower index pipe is coloured red. In a pipe dream with all elbow pipes, all the cells are in configuration $A$, so configuration $A$ is the default configuration for the pipes. If a cell $c$ with a crossing tile and the pipes are in configuration $C$, we say $c$ is the point of \textbf{crossing} for the two pipes going through it; if the pipes are in configuration $D$, then $c$ is the point of \textbf{uncrossing}. If $c$ and $c'$ are points of crossing and uncrossing for two pipes, with the pipes not crossing each other at cells $c\prec b\prec c'$, then we say that $c$ and $c'$ form a \textbf{crossing-uncrossing pair}. The pipes are said to \textbf{inverted} between $c$ and $c'$. Finally, the cells with configuration $B$ are the ones where the filling in the cell is an elbow pipe, but the pipes entering the cell are inverted.
\begin{definition}
    A pipe dream is a \textbf{Go-diagram} if it does not have configuration $B$. A pipe dream is said to be \textbf{reduced} if configuration $D$ never appears. Finally, a \textbf{\reflectbox{L}-diagram} is a reduced Go-diagram. We shall now obtain a $(+,\circ,\bullet)$ filling of $\lambda$ by replacing all cells with configuration $A$ or $B$ with $+$, $C$ with $\circ$, and $D$ with $\bullet$. In this diagram, $\circ$ denotes the points of crossing, i.e., the length of the permutation increases here, while $\bullet$ denotes the points of uncrossing, i.e, the length of the permutation decreases here.
\end{definition}

\begin{remark}
    The definition in the literature for Go-diagram and \reflectbox{L}-diagram is given using the subexpression $\textbf{v}$ of $\textbf{w}_\lambda$ obtained using the reading order on $\lambda$. $D$ is said to be a Go-diagram (resp. \reflectbox{L}-diagram) if $\textbf{v}$ is a distinguished subexpression (resp. positive distinguished subexpression). It can be shown that this definition is independent of the choice of the reading word, and we have defined Go-diagrams and \reflectbox{L}-diagrams in a way that emphasizes this point.
\end{remark}
\begin{example}\label{ex:pipedreams}
    Let $n =6, k=3, \lambda  = (3,3,3)$. In this case $I_\lambda = \{1,2,3\}, w_\lambda = 456123$. Using the reading word in \cref{fig:readingword}, we get the reduced word for $w_\lambda$, as $\textbf{w}_\lambda = s_3s_4s_5s_2s_3s_4s_1s_2s_3$. The permutations corresponding to pipedreams in  \cref{fig:pipedreamex2} and \cref{fig:pipedreamex1} are $v^1 = 123456$ and $v^2 = 124356$  respectively.
    The subexpression corresponding to the pipedream in \cref{fig:pipedreamex2} is $\textbf{v}^1=s_31111111s_3$ and the subexpression for pipedream in \cref{fig:pipedreamex1} is $\textbf{v}^2=s_3111s_3111s_3$,. From \cref{ex:distinguished}, we know that $\textbf{v}^1$ is not a distinguished subexpression, and hence the pipedream in \cref{fig:pipedreamex2} is not a Go-diagram, which can also be seen through the presence of a cell with configuration $B$. By similar reasoning, the pipedream in \cref{fig:pipedreamex1} is a Go-diagram.
\end{example}
\subsection{Grassmannian and Schubert decomposition}        

Let $\mathbb{F}$ denote the field (which will be either $\RR$ or $\mathbb{C}$), and $\mathbb{F}^* = \mathbb{F}\setminus\{0\}$. The \textbf{Grassmannian} $Gr_{k,n}$ is defined as the set of all $k$-dimensional subspaces of $\mathbb{F}^n$. Let $e_1,e_2,...,e_n$ be the standard basis for $\mathbb{F}^n$. Given $V\in Gr_{k,n}$, let $v_1,...,v_k$ be the basis vectors of $V$, and let $v_i=\sum_{j=1}^n (A_V)_{i,j}e_j$. Here, we say the matrix $A_V$ \textbf{realizes} $V$. The set of all $k\times n$ matrices that realize $V$ are row equivalent to each other, so we can also think of the $Gr_{k,n}$ as the set of all full rank $k\times n$ matrices quotiented by the row equivalence relation.

For $I\in \binom{[n]}{k}$, let $\Delta_I(A_V)$ be the maximal minor of $A_V$ with columns given by $I$. These will be referred to as the \textbf{Pl\"{u}cker coordinates} of $V$. Since scaling any row of $A_V$ also gives a matrix realizing $V$, we have an embedding $Gr_{k,n}\to \mathbb{PF}^{\binom{n}{k}-1}$. The Pl\"{u}cker coordinates also satisfy homogeneous polynomial equations called the \textbf{Pl\"{u}cker relations}, giving $Gr_{k,n}$ the structure of a projective variety. If the field is $\mathbb{R}$, the \textbf{totally non-negative Grassmannian} $Gr^{tnn}_{k,n}$ is the set of subspaces $V$ such that all the Pl\"{u}cker coordinates $\Delta_I(A_V)$ have the same sign for any realization $A_V$ of $V$ (i.e. there is a realization of $V$ such that all the Pl\"{u}cker coordinates of $V$ in that realization are all non-negative). If $\Delta_I(A_V)\not=0$ for some realization $A_V$ of $V$, then it is non zero for all realizations $A_V$. So in the statements where the actual value of the coordinate is not relevant, and the only thing being considered is whether it is zero or not, we drop the dependence on the realization, and use $\Delta_I(V)$ to refer to the Pl\"ucker coordinate.

We now look at the simplest decomposition of the Grassmannian: the \textbf{Schubert decomposition}. Let $\lambda$ be a partition and let $I_\lambda$ be the $ k$-element subset corresponding to it. The \textbf{Schubert component} $\Omega_\lambda$ is defined below:
\begin{equation*}
    \Omega_\lambda =\{V\in Gr_{k,n}|I_\lambda \texttt{ is the maximal set such that }\Delta_{I_\lambda}(V)\not=0\}.
\end{equation*}
There is a simple way to see the above decomposition using \textbf{Gaussian elimination}. Since all the matrices realizing $V$ are row equivalent, we have one matrix in the equivalence class that is in row-reduced echelon form. Then $\Omega_\lambda$ is the set of all $V$'s such that the matrix in row reduced echelon form realizing $V$ has pivot columns given by $I_\lambda$. See \cref{ex:schubertex}
\begin{example}\label{ex:schubertex}
    Let $n=6,k=3, I= \{1,3,6\}\in \binom{[6]}{3}$. Then any point in the Schubert cell corresponding to $I$ has the following row-reduced echelon form:
    \begin{equation*}
        A_V=\left[\begin{matrix}
            1&*&0&*&*&0\\
            0&0&1&*&*&0\\
            0&0&0&0&0&1
        \end{matrix}\right]
    \end{equation*}
    where the $*$ entries can take any value in $\mathbb{F}$. Note that the pivot columns of $A_V$ are given by $I$.
\end{example}
Now, if we remove the columns containing 1, and reflect what we have with respect to the vertical axis, the $*$'s form the Young diagram of $\lambda$. Each of these $*$'s could have any value in $\mathbb{F}$, so we have $\Omega_\lambda \simeq \mathbb{F}^{|\lambda|}$. Therefore, we get the following decomposition:
\begin{equation*}
    Gr_{k,n}= \bigcup_{\lambda\in\Lambda_{k,n}}\Omega_\lambda\simeq\bigcup_{\lambda\in\Lambda_{k,n}}\mathbb{F^{|\lambda|}}.
\end{equation*}
These components, along with having a simple topology, also form a stratification, that is:
\begin{equation*}
    \overline{\Omega_\lambda}= \bigcup_{\mu\leq \lambda}\Omega_\mu .
\end{equation*}
Even though this is the coarsest decomposition of $Gr_{k,n}$ that we define here, it is the finest decomposition we know that has both the properties: each component is an affine space, and the components form a stratification.

\subsection{Matroid decomposition}
\begin{definition}
    A \textbf{matroid} $\mathcal{M}$ of \textbf{rank} $k$ is a non-empty subset of $\binom{[n]}{k}$, whose elements which are called \textbf{bases} satisfy the following \textbf{exchange condition}:
    For $I,J\in \mathcal{M},  i\in I\setminus J$ there exists $j\in J$ such that $(I\setminus \{i\})\cup \{j\}\in \mathcal{M}$.
\end{definition}
Given an element $V\in Gr_{k,n}$, let $\mathcal{M}_V\subset \binom{[n]}{k}$ be the set of non-zero Pl\"{u}cker coordinates of $V$. Using Pl\"{u}cker relations, it can be shown that $\mathcal{M}_V$ is a matroid.
\begin{definition}[Matroid stratum]
    Let $\mathcal{M}$ be a matroid of rank $k$. The matroid stratum $\mathcal{S_M}$ of $\mathcal{M}$ is defined as:
    \begin{equation*}
        \mathcal{S}_\mathcal{M}=\{V\in Gr_{k,n}|\mathcal{M}_V=\mathcal{M}\}.
    \end{equation*}
    This gives a stratification of $Gr_{k,n}$, referred to as the \textbf{matroid stratification}, also known as the Gelfand-Serganova stratification.
\end{definition}
The matroids with non-empty matroid stratum are called \textbf{realizable matroids}. The geometry of these strata is highly non-trivial; in fact, they can be as complicated as any algebraic variety \cite{Mnev1988}. The intersection of the matroid decomposition with $Gr_{k,n}^{tnn}$ is more structured; here, it gives $Gr_{k,n}^{tnn}$ a structure of CW-complex \cite{postnikov2006total}. Therefore, the matroids which have a non-empty intersection with $Gr_{k,n}^{tnn}$ are special, and are referred to as \textbf{positroids}.

\subsection{Positroid decomposition}
A \textbf{complete flag} on $\mathbb{F}^n$ is a sequence of subspaces $(V_0,V_1,...,V_n)$ such that 
    \begin{equation*}
        \{0\}=V_0\subset V_1 \subset V_2\subset... V_n = \mathbb{F}^n.
    \end{equation*}
    The \textbf{flag variety} $Fl_n$ is the set of complete flags on $\mathbb{F}^n$. Let $e_1,e_2,..., e_n$ be the standard basis for $\mathbb{F}^n$. The standard flag is $F=(F_0,F_1,...,F_n)$ where $F_i$ the subspace generated by $\{e_1,...,e_i\}$ for all $i\leq n$. We also have a projection $\pi_k:Fl_n\to Gr_{k,n}$, which is simply given by $\pi_k(V_0,V_1,...,V_n)= V_k$.

    Let $G = SL_n(\mathbb{F})$. Let $B^+, B^-, T \subset G$ be the set of upper triangular, lower triangular, and diagonal matrices, respectively. We have homomorphisms $\phi_i:SL_2(\mathbb{F})\to SL_n(\mathbb{F})$  for all $i\leq n-1$:
    \begin{equation*}
        \phi_i\left( \begin{matrix}
            a&b\\
            c&d
        \end{matrix}\right)
        =
        \left( \begin{matrix}
            1& & & & &\\
            &\ddots& & & &\\
            & & a& b& &\\
            & & c& d& &\\
            & & & &\ddots&\\
            & & & & & 1
        \end{matrix}\right)
    \end{equation*}
    where $a$ is the $(i+1)^{th}$ diagonal entry from the bottom.
    The Weyl group in this setting is $W = \mathfrak{S}_n = N_G(T)/T$, where $N_G(T)$ is the normalizer of $T$ in $G$. Let $\Dot{s}_i = \phi_i\left(\begin{matrix}0&-1\\1&0\end{matrix}\right)$. The simple reflections for $W$ are given by $s_i = \Dot{s}_iT$. If $w\in W$ has a reduced word presentation $w = s_{i_1}s_{i_2}...s_{i_m}$, then define $\Dot{w}:= \Dot{s}_{i_1}\Dot{s}_{i_2}...\Dot{s}_{i_m}$. 

    We have $G$-action on $Fl_n$. If $g\in G$, and $(V_0, V_1,...,V_n)\in Fl_n$ then we define $g(V_0,V_1,...,V_n) := (g(V_0),g(V_1),...,g(V_n))$. This action is transitive; moreover, $B^+$ fixes the standard flag. Therefore, we have $Fl_n\simeq G/B^+$. This gives us the following opposite \textbf{Bruhat decompositions} for the flag manifold:
    \begin{equation*}
        G/B^+ = \bigcup_{w\in\mathfrak{S}_n}B^+\Dot{w} B^+ = \bigcup_{v\in\mathfrak{S}_n}B^-\Dot{v} B^+.
    \end{equation*}

    For $v,w\in W$, we define the \textbf{Richardson cell} as $\mathcal{R}_{v,w} = B^+\Dot{w}B^+ \cap B^-\Dot{v}B^+$. This is empty unless we have $v\leq w$ in the Bruhat ordering. Now we are in the position to define the positroid decomposition of the Grassmannian. Let $W^k$ be the set of all Grassmannian permutations with at most one descent, which is at position $k$. Let $w\in W^k$. We define the \textbf{positroid cell} as $\mathcal{P}_{v,w} = \pi_k(\mathcal{R}_{v,w})$. We have the following decomposition of the Grassmannian:
    \begin{equation*}
        Gr_{k,n}(\mathbb{F}) = \bigcup_{w\in W^k}\bigcup_{v\leq w}\mathcal{P}_{v,w}.
    \end{equation*}
    \begin{remark}
        In the literature, positroid cells are also defined as the intersection of cyclic Schubert cells; however, this definition is equivalent to the one we gave above.
    \end{remark}
    \begin{remark}
        The decomposition is called the positroid decomposition because in the case $\mathbb{F=\mathbb{R}}$, when we intersect this with $Gr_{k,n}^{tnn}$, the cells intersect in their full dimension, and we recover the CW complex formed by the matroid decomposition in $Gr_{k,n}^{tnn}$. Furthermore, the positroid cells in $Gr_{k,n}$ are the Zariski closure of the positroid cells in $Gr_{k,n}^{tnn}$ \cite{Knutson_Lam_Speyer_2013}.
    \end{remark}
The Richardson strata $\mathcal{R}_{v,w}$ or their projections $\mathcal{P}_{v,w}$ are stratifications; however, they are no longer a simple product of affine spaces like the Schubert components.

\subsection{Deodhar decomposition}
This decomposition was introduced by Deodhar for the flag variety \cite{deodhar1,deodhar2}. Deodhar's original motivation was to calculate the $R$-polynomials introduced by Kazhdan and
Lusztig \cite{Lusztig1979}. In the setting we are working with a finite field $\mathbb{F}_q$, the number of points inside the Richardson cell $\mathcal{R}_{v,w}$ is given by the $R$-polynomial. Since the Richardson cells themselves are not affine spaces, counting the points is a non-trivial task. Deodhar decomposed each Richardson cell $\mathcal{R}_{v,w}$ into various components $\mathcal{R}_{\textbf{v},\textbf{w}}$ indexed by distinguished subexpressions $\textbf{v}\prec \textbf{w}$, such that each $\mathcal{R}_{\textbf{v},\textbf{w}}\simeq \mathbb{F}^a\times (\mathbb{F}^*)^b$.
\begin{definition}[\cite{deodhar1, deodhar2}]
Let $\textbf{w} = s_{i_1}s_{i_2}...s_{i_m}$ be a reduced word representation of $w$, and let $\textbf{v}$ be the subexpression of $\textbf{w}$. We define the following sets:
\begin{align*}
    J_\textbf{v}^\circ&= \{k\in\{1,...,m\}| v_{(k-1)}<v_{(k)}\}\\
    J_\textbf{v}^\Box&= \{k\in\{1,...,m\}| v_{(k-1)}=v_{(k)}\}\\
    J_\textbf{v}^\bullet&= \{k\in\{1,...,m\}| v_{(k-1)}>v_{(k)}\}.
\end{align*}
\end{definition}
For all $i\leq n-1, p\in \mathbb{F}$, define $x_i(p) = \phi_i\left(\begin{matrix}
    1&p\\0&1
\end{matrix}\right)$ and $y_i(p) =\phi_i\left(\begin{matrix}
    1&0\\ p&1
\end{matrix}\right)$. For a distinguished subexpression $\textbf{v}\prec \textbf{w}$, define the following subset of $G$
\cite{marsh2004parametrizations}:
\begin{equation*}
    G_{\textbf{v},\textbf{w}}:=\left\{\begin{array}{c|l l c}
         & g_l= x_{i_l}(q_l)\Dot{s}_{i_l}^{-1} & \texttt{if }l\in J_\textbf{v}^\bullet& \\
         g= g_1g_2...g_m&g_l = y_{i_l}(p_l) & \texttt{if }l\in J_\textbf{v}^\Box& \texttt{for }p_l\in \mathbb{F}^*, q_l\in \mathbb{F}\\
         & g_l = \Dot{s}_{i_l}&\texttt{if }l\in J_\textbf{v}^\circ&\\
    \end{array}\right\}.
\end{equation*}
Note that $G_{\textbf{v},\textbf{w}}\cong \mathbb{F}^{|J_\textbf{v}^\bullet|}\times (\mathbb{F}^*)^{|J_\textbf{v}^\Box|}$. Marsh and Rietsch proved the following theorem, which gives the parametrization of the Deodhar component. We shall use this theorem to define the Deodhar component.
\begin{theorem}[\cite{marsh2004parametrizations}]
    Let $ \textup{\textbf{v}}$ be a distinguished subexpression of $\textup{\textbf{w}}$. Then,  
    $\mathcal{R}_{\textup{\textbf{v}},\textup{\textbf{w}}} = G_{\textup{\textbf{v}},\textup{\textbf{w}}}B^+$.
\end{theorem}

Deodhar showed the following theorem:
\begin{theorem}[\cite{deodhar1, deodhar2}]
    Let $v,w\in \mathfrak{S}_n$ be two permutations such that $v\leq w$. Then the Richardson cell $\mathcal{R}_{v,w}$ is a disjoint union of Deodhar components: 
    \begin{equation*}
        \mathcal{R}_{v,w} = \bigcup_{\substack{\textup{\textbf{v}}\prec\textup{\textbf{w}}\\\ \pi(\textup{\textbf{v}})=v}}\mathcal{R}_{\textup{\textbf{v}},\textup{\textbf{w}}}.
    \end{equation*}        
\end{theorem}
The Deodhar decomposition for the Grassmannian is obtained similarly to the positroid decomposition: we project the Deodhar decomposition in the flag variety. Letting $w\in W^k$, the Deodhar component in the Grassmannian is given by $\mathcal{D}_{\textbf{v,w}}= \pi_k(\mathcal{R}_{\textbf{v,w}})$. We now have the following decomposition of the Grassmannian due to \cite{kodamawilliams}:
\begin{equation*}
    Gr_{k,n} = \bigcup_{w\in W^k}\left(\bigcup_{\textbf{v}\prec\textbf{w}} \mathcal{D}_{\textbf{v},\textbf{w}}\right).
\end{equation*}
    
\begin{remark}
    The Deodhar decomposition is not a stratification, as the closure of a component is not necessarily a union of components \cite{dudas2008note}.
\end{remark}
The Deodhar decomposition for the flag variety depends on the reduced word $\textbf{w}$; however, in the Grassmannian, the decomposition is only dependent on the permutation $w$ \cite{kodamawilliams}. When we convert a distinguished subexpression into a Go-diagram, the different choices of reduced words for $w$ correspond to different reading words on the Go-diagram. Since characterization of Go-diagrams is not dependent on the choice of reading words, instead of using distinguished subexpressions to index the Deodhar components in $Gr_{k,n}$, we will use Go-diagrams. For a Go-diagram $D$, let $\mathcal{D}_D$ denote the corresponding Deodhar component.

\subsection{Talaska--Williams parametrization of the Grassmannian}
Postnikov gave a parametrization for $Gr_{k,n}^{tnn}$ using the matroid decomposition \cite{postnikov2006total}. To each positroid, he associated a \reflectbox{L}-diagram, and using this \reflectbox{L}-diagram, he constructed a planar network which parametrizes the positroid cell in $Gr_{k,n}^{tnn}$. Talaska--Williams extended this to the entire $Gr_{k,n}$ by replacing the matroid decomposition with the Deodhar decomposition of the Grassmannian, and $\reflectbox{L}$-diagram with Go-diagrams. In this case, the networks they obtain are no longer necessarily planar. 

\begin{definition}
Let $D$ be a Go-diagram inside the partition $\lambda$. We describe network $M_D$ as follows:
\begin{itemize}
    \item We attach an internal vertex to each cell of $\lambda$ which has a $+$ or $\bullet$ in $D$.
    \item We number the south east border of $\lambda$ with numbers from $1$ to $n$ starting from the east most edge. We attach a boundary vertex to each number.
    \item From each internal vertex, draw a horizontal edge towards the right connecting to the nearest internal vertex with a $+$ or a boundary vertex.
    \item From each internal vertex, draw a vertical edge downwards connecting to the nearest internal vertex with a $+$ or a boundary vertex. 
    \item Direct all the horizontal edges from right to left, and all vertical edges from top to bottom. This creates $k$ sources and $n-k$ sinks.
    \item All vertical edges have weight 1. A horizontal edge has weight $a_b$ when the cell $b$, which is the sink of that edge, has a $+$, it has weight $c_b$ when the sink has $\bullet$.
\end{itemize}
The weight matrix $W_D$ of the above network is given as follows:
\begin{equation*}
(W_D)_{st}= (-1)^{b(s,t)}\sum_{P:s\to t} \mathrm{wt}(P).
\end{equation*}
where the sum is over all paths between sources $s$ and $t$, and $b(s,t)$ is the number of sources strictly between $s$ and $t$. We declare a path from any source to itself which has weight 1. See \cref{ex:twnet}
\end{definition}

\begin{example}\label{ex:twnet}
\begin{figure}
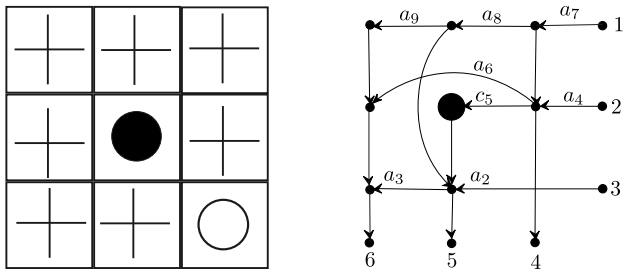

    \centering
    \includesvg[width=35mm, height=35mm]{godiagram.svg} \hspace{10mm}
    \includesvg[width=35mm, height=35mm]{tw4} 
    \caption{Go-diagram and its corresponding network $M_D$.}
    \label{fig:twnetwork}
\end{figure}
Let $D$ be the Go-diagram on the left in \cref{fig:twnetwork}  with the associated network $M_D$ on the right. The weight matrix for $D$ is given by
\begin{equation*}
    W_D=\left[\begin{matrix}
        1&0&0&a_7&a_7a_8&a_7(a_8a_9+a_8a_3+c_5a_3)\\
        0&1&0&-a_4&-c_5a_4&-a_4(a_6+c_5a_3)\\
        0&0&1&0&a_2&a_2a_3
    \end{matrix}\right].
\end{equation*}
\end{example}
Talaska and Williams proved the following theorem:
\begin{theorem}
    Let $\mathcal{W}_D$ be the set of weight matrices obtained by varying the edge weights $a_b$ over $\mathbb{F}^*$, and $c_b$ over $\mathbb{F}$. Then each point in $\mathcal{D}_D$ is realized by a unique matrix in $\mathcal{W}_D$ and each matrix in $\mathcal{W}_D$ realizes a point in $\mathcal{D}_D$.
\end{theorem}

\section{Restricted path parametrization}\label{sec:respaths}

    In this section, we describe the \textbf{restricted path parametrization}. We shall describe paths on the pipe dream corresponding to a Go-diagram following certain rules, and we will use this to define parametrization for the Deodhar component corresponding to the Go-diagram. Later, we will modify the rules on these paths, which will allow us to parametrize the image of $Gr_{k,n}$ under the isomorphism $Gr_{k,n}\simeq Gr_{n-k,n}$. Finally, we will combine these two pictures to describe Pl\"ucker coordinates.
    
    Let $\lambda$ be a partition inside a $k\times (n-k)$ rectangle. Let $D$ be a Go-diagram on $\lambda$. We label the edges of the southeast border of $\lambda$ with the numbers $1$ to $n$, from northeast to southwest. This gives a labelling of the pipes in the pipe dream corresponding to $D$. We label the edges at the northwest border of $\lambda$ with the labels of the pipes exiting at that boundary. Finally, we attach vertices to each of these labels and extend the pipes to connect them to the two vertices (one at the southeast boundary and one at the northwest boundary) that have the same label as the pipe. The vertices on the southeast boundary (resp. north west boundary) will form the sinks (resp. sources) for our restricted paths. \cref{fig:resnetwork} gives an example of this for a particular Go-diagram.
    \begin{definition}[Restricted Paths]\label{def:respaths}
        A \textbf{restricted path} is a path between a source and sink vertex, which starts at a pipe and is allowed to jump between pipes at the cells of $\lambda$. The path has to satisfy one of the following rules when it enters a cell which has $+$ or a $\bullet$:
            \begin{enumerate}
                \item The path follows the pipe it was on.
                \item The path jumps to a pipe with the \textbf{lower} index.
            \end{enumerate}
         A \textbf{jump site} for the restricted path is a cell where the path changes pipes.
          \cref{fig:jumpcases} shows when the possible ways a restricted path can jump (the red coloured pipe has a lower index). See \cref{ex:respaths} for some examples of restricted paths.
          \begin{figure}[h]
        \centering
        \includesvg[width=90mm, height=30mm]{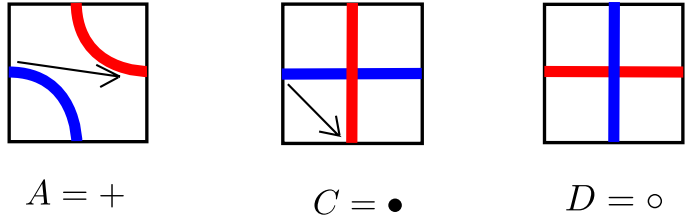}
        \caption{Possible ways for a restricted path to jump.}
        \label{fig:jumpcases}
    \end{figure}
          
    \end{definition}
    Note that we can only have jump sites at cells which have $+$ or $\bullet$. Moreover, a restricted path can jump at a cell only when it enters the cell horizontally.

    \begin{definition}
        At each cell $b$ of the pipe dream $D$, let $\sigma_D(b)=(i,j)$ where $i$ is the index of the pipe exiting through the left boundary of $b$, and $j$ is the index of the other pipe. So at $b$, the restricted path can jump from $i$ to $j$ if there is no white stone at $b$. $\sigma_D(b)$ will be referred to as the \textbf{jump coordinate} of $b$ in $D$. 
        
        To compute the \textbf{weight} of the restricted path, we attach a parameter $\beta_b$ to each cell, and we will refer to this as the \textbf{jump parameter} at $b$. The weight of a restricted path is then defined as the product of the jump parameters at the jump sites of the restricted path.

        The weight matrix of this network is given by:
    \begin{equation*}
        (\Tilde{R}_D)_{st}= \sum_{P:s\to t} \mathrm{wt}(P)
    \end{equation*}
    where the sum runs over all restricted paths starting at source labelled $s$ and sink labelled $t$.

    Let $v_1,v_2,...,v_k$ be the labels that appear on the west boundary from top to bottom.  Let $R_D$ be $k\times n$ matrix defined as follows:
    \begin{equation*}
        (R_D)_{st} =(\Tilde{R}_D)_{v_s t}.
    \end{equation*}
    See \cref{ex:respaths}.
    \end{definition}
     
    \begin{example}\label{ex:respaths}

        \begin{figure}
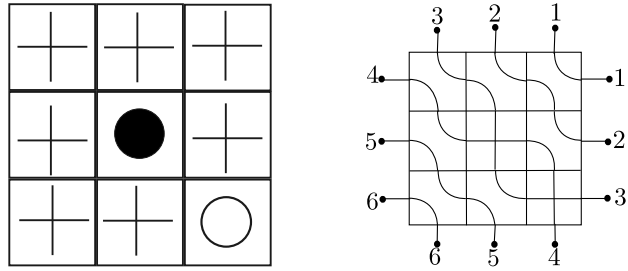

        \centering
        \includesvg[width=35mm, height=35mm]{godiagram.svg} \hspace{10mm}
        \includesvg[width = 35mm, height=35mm]{resnetwork}
        \caption{Go-diagram($D$) and its corresponding pipe dream}
        \label{fig:resnetwork}
    \end{figure}
    
    \begin{figure}
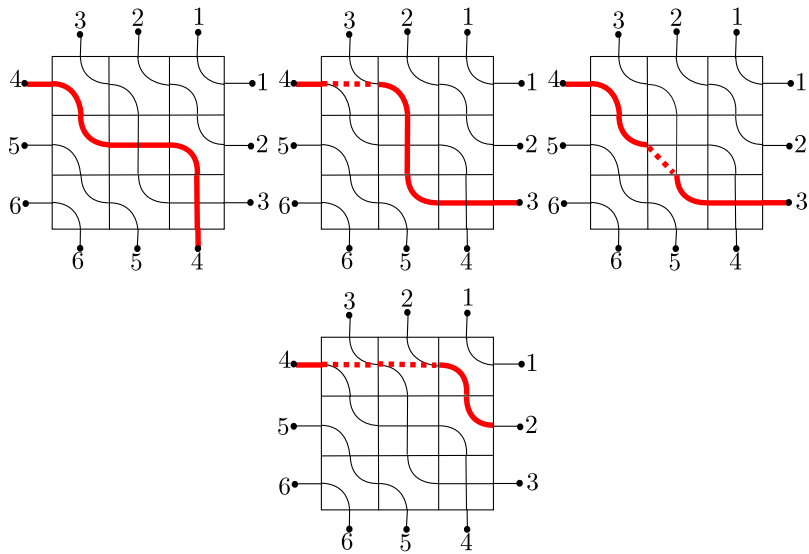

        \centering
        \includesvg[width=35mm, height=35mm]{rp_1}
        \includesvg[width=35mm, height=35mm]{rp_2}
        \includesvg[width=35mm, height=35mm]{rp_3}\\
        \vspace{2mm}
        \includesvg[width=35mm, height=35mm]{rp_4}
        \caption{Examples of restricted paths.}
        \label{fig:respaths}
    \end{figure}
    \begin{figure}
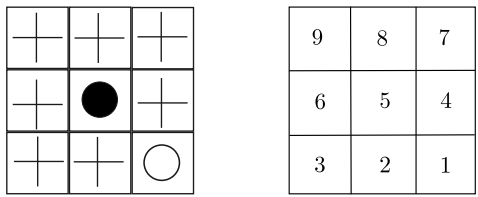

        \centering
        \includesvg[width=25mm, height=25mm]{godiagram.svg} \hspace{10mm}
        \includesvg[width = 25mm, height=25mm]{readingword}
        \caption{Go-diagram $D$ and its corresponding reading word}
        \label{fig:readingword1}
    \end{figure}
        Let $D$ be a Go-diagram given on the left in \cref{fig:resnetwork}. We attach vertices to the labels on the boundary, and extend the pipes to connect to these vertices to get the network shown in the right in \cref{fig:resnetwork}.

        \cref{fig:respaths} gives examples of paths on the above network; the restricted path is red, and the dotted lines denote that the restricted path jumps at that cell. Let $\beta_i$ be the jump parameter at the cell with filling $i$ in the reading word given in \cref{fig:readingword1}. For the restricted paths shown in \cref{fig:respaths}, the weights in the order they appear are $1$, $\beta_9$, $\beta_5$, and $\beta_9\beta_8$.

        The weight matrix $\Tilde{R}_D$ for our example is:
        \begin{equation*}
        \Tilde{R}_D=\left[\begin{matrix}
                1&0&0&0&0&0\\
                \beta_7&1&0&0&0&0\\
                \beta_8\beta_7&\beta_8&1&0&0&0\\
                \beta_9\beta_8\beta_7&\beta_4+\beta_9\beta_8&\beta_9+\beta_5&1&0&0\\
                0&\beta_6\beta_4&\beta_2+\beta_6\beta_5&\beta_6&1&0\\
                0&0&\beta_3\beta_2&0&\beta_3&1
            \end{matrix}\right].
    \end{equation*}
        To get the truncated matrix $R_D$, we take the $4^{th}, 5^{th}$ and the $6^{th}$ row (in this order) and get, 
    \begin{equation*}
        R_D=\left[\begin{matrix}
                \beta_9\beta_8\beta_7&\beta_4+\beta_9\beta_8&\beta_9+\beta_5&1&0&0\\
                0&\beta_6\beta_4&\beta_2+\beta_6\beta_5&\beta_6&1&0\\
                0&0&\beta_3\beta_2&0&\beta_3&1
            \end{matrix}\right].
    \end{equation*}
    
    \end{example}

    The main theorem of this section is that we can parametrize the Deodhar component $\mathcal{D}_D$ using the matrices $R_D$:
    \begin{theorem}\label{resparameterthm}
        Let $\mathcal{R}_D$ be the set of matrices $R_D$ obtained when we let the jump parameters at cells with $\bullet, +$ vary over $\mathbb{F}, \mathbb{F}^*$, respectively, and set the jump parameter at cells with $\circ$ to $0$. Then each matrix in $\mathcal{R}_D$ realizes a point in $\mathcal{D}_D$, and each point in $\mathcal{D}_D$ is realized by a unique matrix in $\mathcal{R}_D$. 
    \end{theorem}
    To prove this theorem, we will first write the parameters in the Talaska--Williams network as bijective functions of our jump parameters, produce the matrix $W_D$ from these parameters, and then show that $W_D$ is row-equivalent to $R_D$. To do this, we first place restricted paths on a grid network, similar to $M_D$, and define an alternative weight function on restricted paths.
    \begin{definition}
        We define the network $N_D$ for a Go-diagram $D$:
        \begin{itemize}
            \item Associate an internal vertex to each cell.
            \item Label the southeast border of the Young diagram with numbers $1$ to $n$. Label the northwest border of the Young diagram with the label of the pipe that exits at the cell.
            \item Associate a boundary vertex to each label on the boundary.
            \item From each internal vertex and boundary vertex on the west boundary, draw an edge right connecting the vertex to the nearest vertex.
            \item From each internal vertex and boundary vertex on the north boundary, draw an edge down connecting the vertex to the vertex.
            \item Direct the horizontal edges left to right, and the vertical edges top to bottom. The boundary vertices on the northwest boundary will form the sources, and the ones on the southeast boundary will form the sinks.
            \item At each internal vertex $b$ on the grid, attach a parameter $\alpha_b$ to the lower-left corner. If $b$ has $+$ on it, put the weight $\frac{-1}{\alpha_b}$ on the upper-right corner.
        \end{itemize}
        We now consider restricted paths as paths on the network $N_D$, where the edges of this grid are viewed as segments of the pipes.
    \end{definition}

    \begin{example}\label{ex:resnet}
        Continuing with the same Go-diagram of \cref{ex:respaths}, \cref{fig:gridresnetwork} shows us the corresponding network $N_D$.
        \begin{figure}[h]
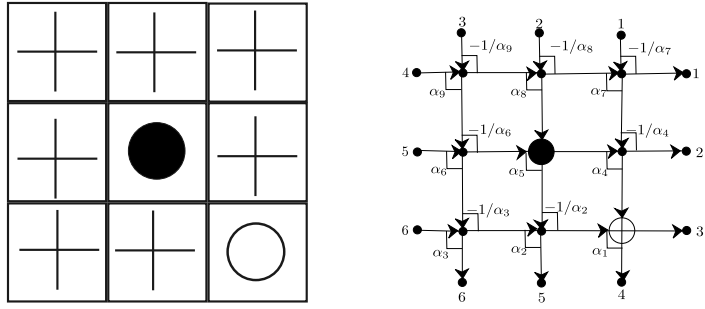

        \centering
        \includesvg[width=40mm, height=40mm]{godiagram.svg} \hspace{10mm}
        \includesvg[width=40mm, height=40mm]{r4}
        \caption{Go-diagram($D$) and its corresponding network $N_D$.}
        \label{fig:gridresnetwork}
    \end{figure}
    We also show how to place two restricted paths from \cref{ex:respaths} on our network $N_D$ in \cref{fig:piperes}.
    \begin{figure}[h]
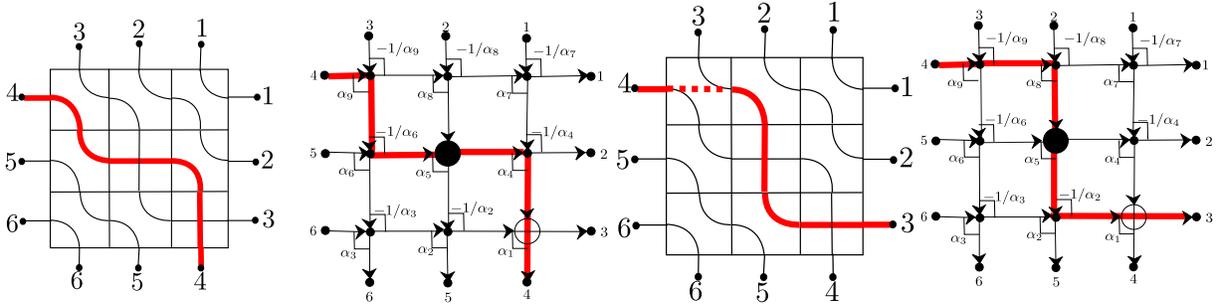

        \centering
        \includesvg[width=80mm, height=40mm]{piperes1}
        \includesvg[width=80mm, height=40mm]{piperes2}
        \caption{Example of two restricted paths placed on $N_D$.}
        \label{fig:piperes}
    \end{figure}
    \end{example}

    We have the following lemma about the restricted paths:
    \begin{lemma}\label{resconstraints}
        Restricted paths on $N_D$ when the enter the cell $b\in \lambda$, they have to follow the constraints below:
        \begin{itemize}
            \item If the restricted path enters $b$ from left, then it can always go right, but down moves are only allowed if $b$ has $+$ or $\bullet$.
            \item If the restricted path enters $b$ from top, then it goes down if the cell has $\circ$ or $\bullet$, otherwise it goes right.
        \end{itemize}
    \end{lemma}
    \begin{proof}
        If the restricted path is entering $b$ from the left, then it is allowed to go right regardless of the entry in the cell. If the cell had $+$, then this corresponds to the restricted path jumping; otherwise, the restricted path is simply following the pipe. The restricted path can also go down if $b$ had $+$ (in this case, it is following the pipe), or if it had $\bullet$ (in this case, it is jumping between pipes).

        When the restricted path enters $b$ from the top, it can never jump; therefore, it goes right if $b$ has $+$, otherwise it goes down.
    \end{proof}
    We introduce a new weight function on the restricted paths, which uses parameters on the network $N_D$. We define $\mathrm{wt}'(P)$ of a restricted path to be the product of all corners that the restricted path traverses.  Moreover, we also extend the definition of restricted paths on $N_D$ as follows:
    \begin{definition}[Restricted paths on $N_D$]
        A restricted path in $N_D$ is a path between two vertices in $N_D$, while following the constraints of Lemma \ref{resconstraints}. If the path starts at an internal vertex, then it can either begin by going down or going right; that is, the first step is unconstrained.
    \end{definition}
    We extend the weight function $\mathrm{wt}'$ for all the restricted paths in $N_D$. Now we write the weight matrix with respect to our new weight function:
    \begin{equation*}
        (\Tilde{S}_D)_{st}= \sum_{P:s\to t} \mathrm{wt}'(P).
    \end{equation*}
    We also define the truncated version of the above matrix; let $v_1,v_2,...,v_k$ be the labels that appear on the west boundary from top to bottom.  Let $S_D$ be  $k\times n$ matrix defined as follows:
    \begin{equation*}
        (S_D)_{st} =(\Tilde{S}_D)_{v_s t}.
    \end{equation*}
    \begin{example}
        Continuing \cref{ex:resnet}, let $P_1$ be the restricted path on the left (the one that corresponds to the pipe in the pipe dream), and let $P_2$ be the other path. From \cref{ex:respaths}, we have $\mathrm{wt}(P_1) = 1, \mathrm{wt}(P_2) = \beta_9$. Our new weight function on these paths is:
        \begin{align*}
            \mathrm{wt}'(P_1) &= \alpha_9\cdot\frac{-1}{\alpha_6}\cdot \alpha_4\\
            \mathrm{wt}'(P_2)&= \alpha_8\cdot \frac{-1}{\alpha_2}
        \end{align*}
        For our example, the weight matrix is given by:
    \begin{equation*}
            \Tilde{S}_D=\left[\begin{matrix}
                -\frac{1}{\alpha_7}&0&0&0&0&0\\
                -\frac{1}{\alpha_8}&\frac{\alpha_7}{\alpha_8\alpha_4}&0&0&0&0\\
                -\frac{1}{\alpha_9}&-\frac{\alpha_7}{\alpha_9}&\frac{\alpha_8}{\alpha_9\alpha_2}&0&0&0\\
                1&-\frac{\alpha_9}{\alpha_6}-\frac{\alpha_7}{\alpha_4}&\frac{\alpha_9\alpha_5}{\alpha_6\alpha_2}-\frac{\alpha_8}{\alpha_2}&-\frac{\alpha_9\alpha_4}{\alpha_6}&0&0\\
                0&1&-\frac{\alpha_6}{\alpha_3}-\frac{\alpha_5}{\alpha_2}&\alpha_4&-\frac{\alpha_6\alpha_2}{\alpha_3}&0\\
                0&0&1&0&\alpha_2&\alpha_3
            \end{matrix}\right].
    \end{equation*}
    and the truncated matrix is:
    \begin{equation*}
            S_D=\left[\begin{matrix}
                1&-\frac{\alpha_9}{\alpha_6}-\frac{\alpha_7}{\alpha_4}&\frac{\alpha_9\alpha_5}{\alpha_6\alpha_2}-\frac{\alpha_8}{\alpha_2}&-\frac{\alpha_9\alpha_4}{\alpha_6}&0&0\\
                0&1&-\frac{\alpha_6}{\alpha_3}-\frac{\alpha_5}{\alpha_2}&\alpha_4&-\frac{\alpha_6\alpha_2}{\alpha_3}&0\\
                0&0&1&0&\alpha_2&\alpha_3
            \end{matrix}\right].
    \end{equation*}
    \end{example}

    We will now describe how to transform the $\alpha_b$ parameters into the parameters for $M_D$. In the network $M_D$, for a cell $b\in \lambda$, let $\chi(b)$ be the nearest internal vertex with a $+$ or boundary vertex right of $b$. Let $\rho(b)$ be the number of sources strictly between the source boundary vertex, which is in the same row as $b$ and the sink boundary vertex, which is in the same column as $b$. Now we define the parameters for $M_D$ in terms of our corner weight parameters in $N_D$:
    \begin{align}
        a_b&= (-1)^{\rho(b)-\rho(\chi(b))}\frac{\alpha_b}{\alpha_{\chi(b)}} & \texttt{if }b\;\texttt{has }+\nonumber\\
        c_b&= (-1)^{\rho(b)-\rho(\chi(b))}\frac{\alpha_b}{\alpha_{\chi(b)}} & \texttt{if }b\;\texttt{has }\bullet\label{eqn:twparatransform}
    \end{align}
    where we put $\alpha_{\chi(b)} = 1$ and $\rho(\chi(b))=0$ if $\chi(b)$ is a boundary vertex. Note that $c_b = 0\iff \alpha_b = 0$ if $b= \bullet$ and $a_b\not=0$ if $b=+$, so we have $c_b\in \mathbb{F}, a_b\in \mathbb{F}^*$. We can also evaluate $\alpha_b$'s in terms of $a_b$'s and $c_b$'s recursively, starting with cells $b$ on the southeast boundary of $\lambda$. So \cref{eqn:twparatransform} gives a continuous bijection between the parameters.
    \begin{lemma}\label{twtransform}
        Let $\mathcal{S}_D$ be the set of matrices $S_D$ obtained when we let the $\alpha_b$ at cells with $\bullet, +$ vary over $\mathbb{F}, \mathbb{F}^*$, respectively, and setting the $\alpha_b$ at cells with $\circ$ to $0$. Then each matrix is row equivalent to a unique matrix in $\mathcal{W}_D$.
    \end{lemma}
    The proof of the above lemma uses induction on $k$, while keeping $n-k$ constant. Note that the matrices in $\mathcal{W}_D$ are in row-reduced echelon form, so to get a unique representative of $S_D$ in $\mathcal{W}_D$, we row-reduce it!  When we row reduce $S_D$ to get $W_D$, the entries of the first row after the row reduction can be described using Pl\"{u}cker coordinates of $S_D$. Using Lindstr{\"o}m-Gessel-Viennot lemma \cite{lgv}, we can write these entries as a sum of non-intersecting restricted paths. Since in $W_D$ these entries correspond to paths on the network $M_D$, the lemma below produces a weight-preserving bijection if the parameters in both $M_D$ and $N_D$ are connected using \cref{eqn:twparatransform}. In the following lemma, we use $\mathrm{wt}_{TW}$ to refer to the weights on the paths on the network $M_D$. For the lemma below, we number the rows from $1$ to $k$, starting from the top, and the columns from $1$ to $n-k$, starting from the left.
    \begin{lemma}\label{bijection}
        Let the labels for the sinks in $M_D$ given by $u_1<u_2<...<u_{n-k}$. Let $v_1<v_2<...<v_k$ be the labels for the sources in $M_D$. Let $h\in \{u_1, u_2,...,u_{n-k}\}, h>v_1$. Let $\mathcal{P}_h$ be the set of non-intersecting systems of restricted paths that start at the sources on the west boundary of $N_D$, and the sinks are given by $\{v_2,v_3,...,v_k, h\}$. Let $P = (p_1, p_2,..., p_k)\in \mathcal{P}_h$, $p_i$ is a restricted starting from the source in row $i$ from the top and $(d_1,d_2,...,d_k)$ be the corresponding sinks for $P$ ($d_i$ is the sink for $p_i$). Let $\tau_P \in \mathfrak{S}_k$ be the permutation such that $d_{\tau_P(i)} = v_i$ for $1<i\leq k$, and $d_{\tau_P(1)} = h$. Let $\varrho$ be the number of sources strictly between $v_1$ and $h$ in $M_D$, and $\mathcal{Q}_h$ is the set of all paths from $v_1$ to $h$ in $M_D$. 
        
        Now if the parameters for $M_D$ are obtained from parameters in $N_D$ using \cref{eqn:twparatransform}, then there is a bijection $f:\mathcal{P}_h\to \mathcal{Q}_h$ such that 
        \begin{equation}\label{eqn:weightbijection}
            \mathrm{sign}(\tau_P)\mathrm{wt}'(P) = (-1)^\varrho\mathrm{wt}_{TW}(f(P))
        \end{equation}
        for all $P\in \mathcal{P}_h$
    \end{lemma}
    \begin{proof}
        For $P=(p_1,...,p_k)$, we construct $f(P) = Q$ row wise. Any path in $\mathcal{Q}_h$ can be uniquely specified by the columns that it goes down in each row, and it will go down exactly $\varrho+1$ times. So given $P$, we are going to give a sequence $(t_1,...,t_{\varrho+1})$, $t_i\geq t_{i+1}$, and $Q$ goes down in column $t_i$ in row $i$. We will also calculate $\tau_P$ during this algorithm. We start by setting $\tau_P =\epsilon$. Now, since $p_1$ doesn't go to $v_1$, it has to go down in some column, and we assign $t_1$ to this column. At each row, the process will be similar; there is exactly one restricted path in our non-intersecting system of restricted paths which goes down from row $i$ to row $i+1$, and we will define $t_{i}$ to be the column that our unique restricted path uses to go down.

        We assume $t_i$ has been calculated, and we will calculate $t_{i+1}$. Since $v_2,...,v_{i}$ need to be sinks, exactly one of the restricted paths $p_1, p_2,...,p_i$ is coming down to the $(i+1)^{th}$ row. Let this restricted path be $p$, and we have chosen $t_i$ such that $p$ comes down to row $i+1$ in column $t_i$ ($p = p_1$ at the start of the algorithm). Now we have two cases:
        \begin{enumerate}
            \item There is a $+$ in row $i+1$ and column $t_i$. In this case, $p$ is forced to go right. Since $p_{i+1}$ can't intersect $p$, so it can't reach $v_{i+1}$. Moreover $p_{i+2},..., p_{k}$ can't reach $v_{i+1}$ as restricted paths can't go up. So $v_{i+1}$ is the sink for $p$, and $p_{i+1}$ is forced to go down at some column $t_{i+1} \leq t_i$. Note that since we have a $+$ in a row $i+1$ and column $t_i$, $Q$ can go left from this cell. Also since $p_{i+1}$ can only go down at either a $+$ or $\bullet$, so $Q$ can go down at column $t_{i+1}$ in row $i+1$. Let $p = p_l$. We update $p = p_{i+1}$ and we swap $i+1$ and $l$ in $\tau_P$.
            \item There is a $\circ$ or $\bullet$ in row $i+1$ and column $t_i$. In this case, $p$ is forced to go down, and since $v_{i+1}$ needed to be a sink for the restricted path, and only $p_{i+1}$ can reach it, $p_{i+1}$ goes right horizontally to $v_{i+1}$. In this case we set $t_{i+1}=t_i$, we don't change $p$ or $\tau_P$.
        \end{enumerate}
        This process stops when we reach the $(\varrho+1)^{th}$ row, and the restricted path goes to sink $h$. Clearly, $Q$ produced this way is a path from $v_1$ to $h$ in $M_D$.

        Now given $Q\in \mathcal{Q}_h$, we shall produce $P$ such that $P\in \mathcal{P}_h$. Let $t_1\geq t_2\geq ...\geq t_{\varrho+1}$ be such the $Q$ goes down in column $t_i$ at row $i$. We force $p_1$ to go down in column $t_1$. We shall move row-wise, starting at the second row, and at each row, we will have found one of the restricted paths in $P$ completely, and one restricted path will have been partially evaluated. Let us assume that we are in row $i+1$, and $i-1$ restricted paths in $P$ have been evaluated. Let $p$ be the partially evaluated path coming down to row $i+1$ in column $t_i$(for $i=1, p=p_1$, and no restricted path has been evaluated at this stage). In row $i+1$, we again have two cases:
        \begin{enumerate}
            \item There is a $+$ in row $i+1$ and column $t_i$. In this case, we force $p$ to go to the sink $v_{i+1}$, hence completing the restricted path. We start evaluating $p_{i+1}$, and we make it go down in column $t_{i+1}$. Since $Q$ could only go down if we are at a cell with $+$, or $\bullet$, $p_{i+1}$ can go down at column $t_{i+1}$. We update $p = p_{i+1}$.
            \item There is a $\circ$ or $\bullet$ in row $i+1$ and column $t_i$. In this case $Q$ is forced go down, so $t_{i+1}= t_i$. We continue extending the restricted path $p$ downwards, and we send $p_{i+1}$ directly to $v_{i+1}$.
        \end{enumerate}
        We will stop at row $\varrho+1$, when $Q$ reaches $h$. The path that is $p$ at this point is the one that has $h$ as its sink. Restricted paths $p_{\varrho+1}, ...., p_k$ are sent horizontally to $v_{\varrho+1},...,v_k$ respectively. It is clear that $f(P) =Q$, and hence $f$ is a bijection.
        
        Now we will show \cref{eqn:weightbijection}. We note that the weight of $P$ or $Q$ can be calculated row-wise again. For $i>1$, let $I_i\in \{1,2\}$ denote the case that was triggered to calculate $t_{i}$, let $b_{1,i}$ and $b_{2,i}$ denote the cells in row $i$ and column $t_i$ and $t_{i-1}$ respectively. Let $b_1$ denote the cell in first row and column $t_1$ For $1<i\leq \varrho+1$, define $r_i$ as follows:
        \begin{equation*}
	   r_i=\left\{\begin{matrix}
			1&\texttt{if }I_i=2\\
			-\frac{\alpha_{b_{1,i}}}{\alpha_{b_{2,i}}}&\texttt{if }I_i=1
            \end{matrix}\right.
        \end{equation*}
        We define $r_1 = \alpha_{b_1}$. Clearly $\mathrm{wt}'(P) = \prod_{i=1}^{\varrho+1}r_i$. For $Q$, we define $q_i$ for $1<i\leq \varrho+1$ as follows:
        \begin{equation*}
	   q_i=\left\{\begin{matrix}
			1&\texttt{if }t_i=t_{i-1}\\
			a_{b_{1,i}}\prod_ba_b&\texttt{if }b_{1,i} = + \texttt{ and if }t_i\not=t_{i+1}\\
                c_{b_{1,i}}\prod_ba_b&\texttt{if }b_{1,i} = \bullet\texttt{ and if }t_i\not=t_{i+1}\\
            \end{matrix}\right.
        \end{equation*}
        where $b$ are cells strictly between $b_{1,i}$ and $b_{2,i}$ in row $i$. Now if put the $a_b$'s and $c_b$'s in terms of $\alpha_b$'s using \cref{eqn:twparatransform}, we get a telescoping product:
        \begin{equation*}
	   q_i=\left\{\begin{matrix}
			1&\texttt{if }I_i=2\\
			(-1)^{\rho(b_{1,i})-\rho(b_{2,i})}\frac{\alpha_{b_{1,i}}}{\alpha_{b_{2,i}}}&\texttt{if }I_i=1
            \end{matrix}\right.
        \end{equation*}
        For the first row weight in $Q$, again in terms of $\alpha_b$'s we have a telescoping product, and we can assign $q_1 = (-1)^{\rho(b_1)}\alpha_{b_1}$ to get $\mathrm{wt}_{TW}(Q) = \prod_{i=1}^{\varrho+1}q_i$. Since $\mathrm{wt}'(P)$ (resp. $\mathrm{wt}_{TW}(Q)$) is a ratio of two monomials in $\alpha_b$'s, let $|\mathrm{wt}'(P)|$ (resp. $|\mathrm{wt}_{TW}(Q)|$) be that ratio where if have removed the sign of $\mathrm{wt}'(P)$ (resp. $\mathrm{wt}_{TW}(P)$). The discussion above gives us $|\mathrm{wt}'(P)| = |\mathrm{wt}_{TW}(Q)|$.

        Now if $c$ is the number of $i>2$ for which we have $I_i=1$, then $(-1)^c|\mathrm{wt}'(P)| = \mathrm{wt}'(P)$. Moreover, whenever $I_i=1$, we multiply $\tau_P$ by a transposition. So $\mathrm{sign}(\tau_P) = (-1)^c$. This gives us $\mathrm{sign}(\tau_P)\mathrm{wt}'(P) =|\mathrm{wt}'(P)|$. Now we define
        \begin{equation*}
            d= \rho(b_1) +\sum_{1<i\leq \varrho+1} \rho(b_{1,i})-\rho(b_{2,i}).
        \end{equation*}
        Since if $I_i =2$, $t_i = t_{i-1}$, we have $(-1)^d|\mathrm{wt}_{TW}(Q)| = \mathrm{wt}_{TW}(Q)$. Since every row introduces a source in $M_D$, we have $\rho(b_{1,i})-\rho(b_{2,i+1}) = 1$ for all $i>2$, and $\rho(b_1)- \rho(b_{2,2}) = 1$, and $\rho(b_{1,\varrho+1}) = 0$.
        \begin{align*}
            d &= \rho(b_1) +\sum_{1<i\leq \varrho+1} \rho(b_{1,i})-\rho(b_{2,i})\\
            &= \rho(b_1)-\rho(b_{2,2})+\sum_{i=2}^{\varrho}\rho(b_{1,i})-\rho(b_{2,i+1}) = \varrho
        \end{align*}
        giving us $(-1)^\varrho\mathrm{wt}_{TW}(Q) = |\mathrm{wt}_{TW}(Q)|$, proving equation \eqref{eqn:weightbijection}.
    \end{proof}
    \begin{example}
        Continuing \cref{ex:resnet}, here we have $v_1=1, v_2=2,$ and $v_3 = 3$. Let $h=6$. \cref{fig:bijection} illustrates the above bijection for our running example.
    \begin{figure}
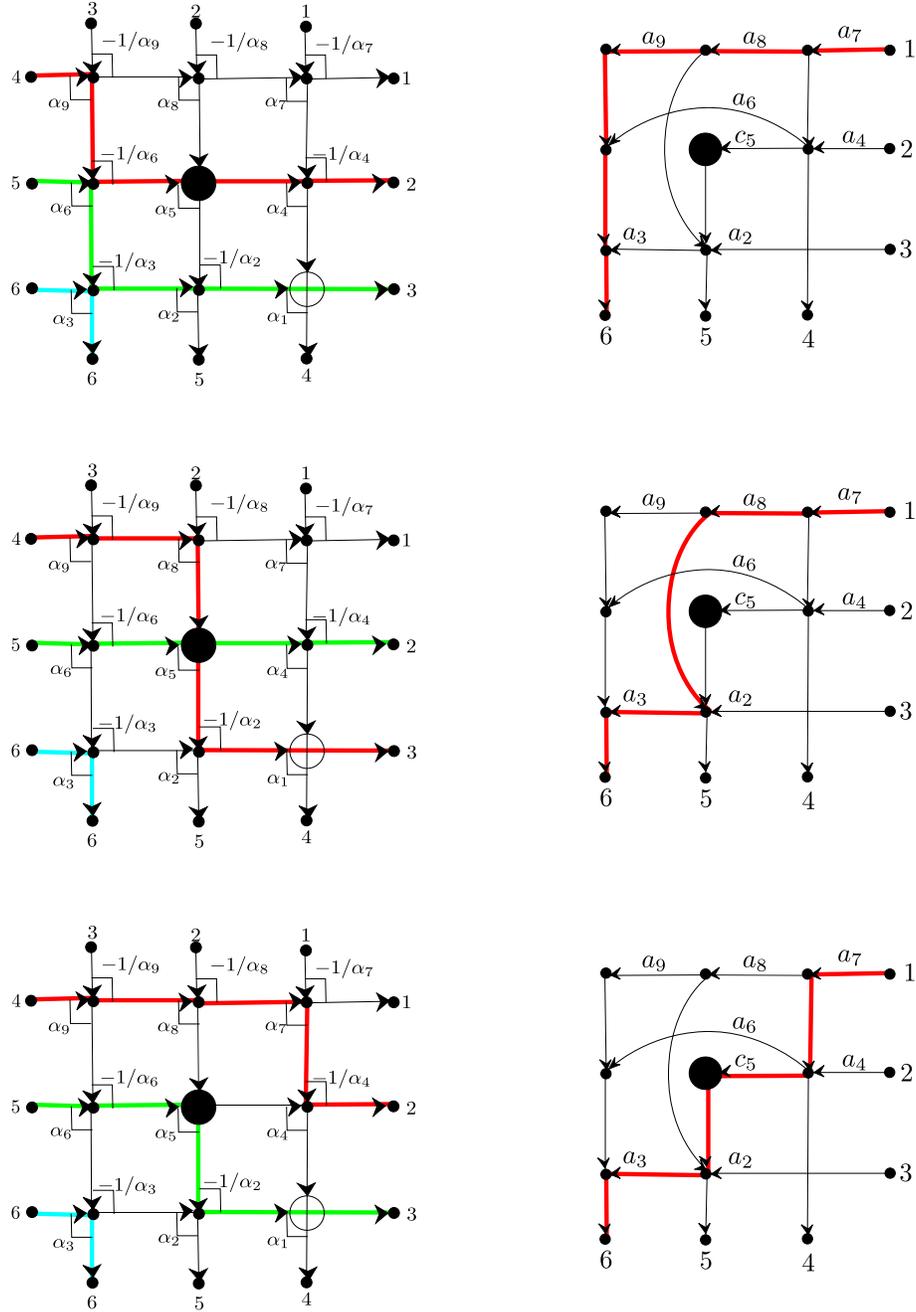

        \centering
        
        \includesvg[width=120mm, height=60mm]{parabijection1}

        \vspace{10mm}
        
        \includesvg[width=120mm, height=60mm]{parabijection2}
        
        \vspace{10mm}
        
        \includesvg[width=120mm, height=60mm]{parabijection3}
        \caption{Bijection described in \cref{bijection}}
        \label{fig:bijection}
    \end{figure}
    \end{example}
    
    \begin{proof}[Proof of \cref{twtransform}]
        We show that $W_D$ obtained from $\alpha_b$ parameters using equation \eqref{eqn:twparatransform} is row equivalent to $S_D$. We first note that $W_D$ is in row-reduced echelon form, whereas $S_D$ is not. Therefore, we will obtain the row-reduced echelon form of $S_D$. 
        
        We start with the first row. Let $s_i$ denote the $i^{th}$ row of $S_D$, let $\{e_i\}_{i\leq n}$ denote the standard basis for $\mathbb{F}^n$, and let $\{\overline{e}_i\}_{i\leq n}$ denote the basis dual to it. Let the labels for the sinks in $M_D$ given by $u_1<u_2<...<u_{n-k}$. Let $v_1<v_2<...<v_k$ be the labels for the sources in $M_D$. We define:
        \begin{equation}\label{rowtransform}
	s_1':=\sum_{i=1}^k\left((-1)^{(i-1)}s_i\cdot\left(\bigwedge_{j=2}^k\overline{e}_{v_j}\right)(s_1,s_2,...,s_{i-1}, s_{i+1},..., s_k)\right).
        \end{equation}
        We have for any $h\leq n$,
        \begin{align*}
        	\overline{e}_h(s_1')&= \sum_{i=1}^k\left((-1)^{(i-1)}\overline{e}_h(s_i)\cdot\left(\bigwedge_{j=2}^k\overline{e}_{v_j}\right)(s_1,s_2,...,s_{i-1}, s_{i+1},...,s_k)\right)\\
        	&=\overline{e}_h\wedge\left(\bigwedge_{j=2}^k\overline{e}_{v_j}\right)(s_1,s_2,...,s_k)
        \end{align*}
	where the last equation is obtained by noting that the sum is just the Laplace expansion of the determinant along the first row. Therefore, the entries of $s_1'$ are maximal minors of the matrix $S_D$.
        Now if $h\in \{v_2,...,v_k\}$ then we get $\overline{e}_h(s_1') = 0$. If $h = v_1$, then we have 
        \begin{equation*}
            \overline{e}_{v_1}(s_1')  = \left(\bigwedge_{j=1}^k\overline{e}_{v_j}\right)(s_1,s_2,...,s_k)=1
        \end{equation*}
        as restricted paths can't go up and completely horizontal restricted paths have $\mathrm{wt}'$ as 1, which means that the minor on the right side of the equation is the determinant of an identity matrix. For all $h<v_1$, $\overline{e}_h(s_1')$ is going to 0 as the $h^{th}$ column of $S_D$ is all 0. So in the row-reduced echelon form of $S_D$, $s_1'$ is going to the first row. If $w_1$ is the first row of $W_D$, then it remains to show that all the non-trivial entries of $w_1$ are the same as those of $s_1'$. In particular, we shall show $\overline{e}_h(s_1') = \overline{e}_h(w_1)$ for all $h\in \{u_1,..., u_{n-k}\}, h>v_1$.

        We can expand the minor in $\overline{e}_h(s_1')$ using Lindstr{\"o}m-Gessel-Viennot lemma as a sum of non-intersecting restricted paths. Let $\mathcal{P}_h$ be the set of all non-intersecting paths that show up in the expansion of the minor. If $P = (p_1, p_2,..., p_k)\in \mathcal{P}_h$ and $(d_1,d_2,...,d_k)$ be the corresponding sinks for $P$ ($d_i$ is the sink for $p_i$), then let $\tau_P$ be the permutation such that $d_{\tau_P(i)} = v_i$ for $i>2$, and $d_{\tau_P(1)} =h$. So by Lindstr{\"o}m-Gessel-Viennot lemma we have:
        \begin{equation*}
            \overline{e}_h(s_1') = \sum_{P\in \mathcal{P}_h} \mathrm{sign}(\tau_P)\mathrm{wt}'(P)
        \end{equation*}
        where we have $\mathrm{wt}'(P)$ as the product of the weight of all restricted paths in $P$. If $\varrho$ is the number of sources of strictly between $v_1$ and $h$ in $M_D$ and $\mathcal{Q}_h$ is the set of all paths from $v_1$ to $h$ in $M_D$, then we have 
        \begin{equation*}
            \overline{e}_h(w_1) = \sum_{Q\in \mathcal{Q}_h}(-1)^\varrho\mathrm{wt}_{TW}(Q)
        \end{equation*}
        where $\mathrm{wt}_{TW}$ is the weight of paths $M_D$. Then, using \cref{bijection}, we get $w_1 = s_1'$.

        Since we have shown that the first row of $W_D$ is the same as that of the row-reduced version of $S_D$, we can finish the rest of the proof using induction on $k$, while keeping $n-k$ constant. $k=1$ is done already, so let $k>1$. Let $D'$ be the diagram obtained by removing the first row of $D$. $D'$ is still a Go-diagram, now inside a $(k-1)\times (n-k)$ rectangle. In the construction of $M_{D'}$ or $N_{D'}$, we remove the boundary vertices labelled 1 and decrease the labels of other boundary vertices by 1. Note that the equations converting $a_b,c_b$'s to $\alpha_b$'s in \cref{eqn:twparatransform} were only dependent on the rows. So we can use the same equations to get $W_{D'}$. We have the following:
        \begin{equation*}
            \begin{matrix}

	       S_D= \left[\begin{matrix}
			    \multicolumn{2}{c}{s_1}\\
			0^{[k-1]\times 1} & S_{D'}
		  \end{matrix}\right],
            &

	       W_D= \left[\begin{matrix}
			\multicolumn{2}{c}{w_1}\\
			0^{[k-1]\times 1} & W_{D'}
		  \end{matrix}\right]

            \end{matrix}.
        \end{equation*}
        We do the row transformation of the first row of $S_D$ according to equation \eqref{rowtransform} to get:
        \begin{equation*}
            S_D'= \left[\begin{matrix}
			    \multicolumn{2}{c}{s_1'}\\
			0^{[k-1]\times 1} & S_{D'}
		  \end{matrix}\right].
        \end{equation*}
         Now, using the induction hypothesis, we get that $S_{D'}$ is row equivalent to $W_{D'}$, hence $S_D'$ is row equivalent to $W_D$, finally giving us that $S_D$ itself is row equivalent to $W_D$.
    \end{proof}
    We now describe how the two weight functions, $\mathrm{wt}$ and $\mathrm{wt}'$ on restricted paths are related to each other.
    \begin{definition}
        For pipe $p$ in the pipe dream, and two cells $b$ and $b'$ ($b'\in b_{in})$ that $p$ passes through, we define $p^{b,b'}$ as the restricted path that connects the internal vertices of $b$ and $b'$ and coincides with $p$ between $b$ and $b'$, that is, it is the segment of $p$ between $b$ and $b'$.

        If the index of $p$ is $i$, then we denote $p^{b,0}$ as the restricted path, which is the segment of $p$ connecting the internal vertex of $b$ with the boundary vertex on the southeast boundary indexed by $i$. We also denote $p^{0,b'}$ as the restricted path, which is the segment of $p$ connecting the boundary vertex $i$ on the northwest boundary with the internal vertex at $b'$.

        We shall attach a label to each $p^{b,b'}$ (or $p^{b,0}$, $p^{0,b'}$), which is going to be the same as the index of $p$.
    \end{definition}
    For a cell $b$, define $\prescript{-}{}{b}$ as the lower index pipe passing through $b$, and let $\prescript{+}{}{b}$ be the higher index pipe passing through $b$.
    \begin{lemma}\label{wttransform}
        Let $b$ be a cell in $D$. Let $t_b$ be:
        \begin{equation*}
            t_b = \left\{\begin{matrix}
            \frac{1}{\alpha_b}&\texttt{if\;} b \texttt{\; has\;}+\\
            \alpha_b&\texttt{if\;} b \texttt{\; has\;}\bullet
            \end{matrix}\right.
        \end{equation*}
        and let 
        \begin{equation}\label{eqn:jumptransform}
            \beta_b = \frac{t_b\mathrm{wt}'((\prescript{-}{}{b})^{b,0})}{\mathrm{wt}'((\prescript{+}{}{b})^{b,0})}
        \end{equation}
        Let $r$ be a restricted path from source $i$ to sink $j$, and let $s$ refer to pipe indexed $i$, then we have the following:
        \begin{equation}\label{eqn:restransform}
            \mathrm{wt}(r) = \frac{\mathrm{wt}'(r)}{\mathrm{wt}'(s)}
        \end{equation}
    \end{lemma}
    \begin{proof}
        Let the jump sites of $r$ be given by $b_1, b_2,...,b_l$, where $b_c\in (b_d)_{in}$ for all $d<c$. Therefore we have $\prescript{+}{}{b_1} = s$, and $\prescript{+}{}{b_i}= \prescript{-}{}{b_{i-1}}$. Define the following:
        \begin{equation*}
            u(b) = \left\{\begin{matrix}
            \frac{1}{\alpha_b}&\texttt{if\;} b \texttt{\;has\;}+\\
            1&\texttt{if\;} b \texttt{\;has\;}\bullet
            \end{matrix}\right.
        \end{equation*}
        \begin{equation*}
            v(b) = \left\{\begin{matrix}
            1&\texttt{if\;} b \texttt{\; has\;}+\\
            \alpha_b&\texttt{if\;} b \texttt{\; has\;}\bullet
            \end{matrix}\right.
        \end{equation*}
        Then we have the following:
        \begin{equation*}
            \mathrm{wt}'(r)=\mathrm{wt}'(s^{0,b_1})\cdot \left(\prod_{m=1}^{l-1}v(b_m)\mathrm{wt}'((\prescript{-}{}{b_m})^{b_m,b_{m+1}})\right)\cdot v(b_l)\mathrm{wt}'((\prescript{-}{}{b_l})^{b_l,0})
        \end{equation*}
        We also have the following for all $m\geq 0$ (use $b_0 =0$ and $(\prescript{-}{}{b_0})^{0,0} = s$):
        \begin{equation*}
            \mathrm{wt}'((\prescript{-}{}{b_m})^{b_m,b_{m+1}})=u(b_{m+1})\frac{\mathrm{wt}'((\prescript{-}{}{b_m})^{b_m,0})}{\mathrm{wt}'((\prescript{-}{}{b_m})^{b_{m+1},0})}.
        \end{equation*}
        Using the above, we get:
        \begin{align*}
            \mathrm{wt}'(r)&=\mathrm{wt}'(s^{0,b_1})\cdot \left(\prod_{m=1}^{l-1}v(b_m)\mathrm{wt}'((\prescript{-}{}{b_m})^{b_m,b_{m+1}})\right)\cdot v(b_l)\mathrm{wt}'((\prescript{-}{}{b_l})^{b_l,0})\\
             &= u(b_{1})\frac{\mathrm{wt}'(s)}{\mathrm{wt}'(s^{b_{1},0})}\cdot\left(\prod_{m=1}^{l-1}v(b_m)u(b_{m+1})\frac{\mathrm{wt}'((\prescript{-}{}{b_m})^{b_m,0})}{\mathrm{wt}'((\prescript{-}{}{b_m})^{b_{m+1},0})}\right)\cdot v(b_l)\mathrm{wt}'((\prescript{-}{}{b_l})^{b_l,0})\\
            &=\mathrm{wt}'(s)\cdot\prod_{m=1}^{l}v(b_m)u(b_{m})\frac{\mathrm{wt}'((\prescript{-}{}{b_m})^{b_m,0})}{\mathrm{wt}'((\prescript{+}{}{b_m})^{b_{m},0})}\\
            &=\mathrm{wt}'(s)\cdot\prod_{m=1}^{l}t_{b_m}\frac{\mathrm{wt}'((\prescript{-}{}{b_m})^{b_m,0})}{\mathrm{wt}'((\prescript{+}{}{b_m})^{b_{m},0})}\\
            &=\mathrm{wt}'(s)\cdot\prod_{m=1}^{l}\beta_{b_m}\\
            &= \mathrm{wt}'(s)\cdot \mathrm{wt}(r).
        \end{align*}
        Now, since there is a unique restricted path from source $i$ to $i$ which is $s$ itself, and this path never uses the parameters $\alpha_b$ where $b = \bullet$ in $D$, then we know that $\mathrm{wt}_2(s)\not=0$, as $\alpha_b\not=0$ for all cells $b$ with $+$. This proves \cref{eqn:restransform}.
    \end{proof}
    \begin{proof}[Proof of Theorem \ref{resparameterthm}]
        We will show that $S_D$ is row equivalent to $R_D$ if the $\beta_b$'s are obtained from $\alpha_b$'s using \cref{eqn:jumptransform}. First we note that if $b=\bullet$, then $\beta_b = 0 \iff \alpha_b$, and if $b=+$, then $\beta_b\not=0$, as the weights of $\prescript{-}{}{b}^{b,0}$ and $\prescript{+}{}{b}^{b,0}$ is non zero. Moreover, we can solve for $\alpha_b$ parameters in terms of $\beta_b$ parameters recursively, starting with cells on the southeast boundary of $\lambda$.

        Now in $\Tilde{S}(D)$, if we rescale all the rows by dividing all the entries in that row by the diagonal entry in that row, \cref{wttransform} shows that we get $\Tilde{R}(D)$. Now $R_D$ and $S_D$ are obtained by choosing the same rows in fixed order from $\Tilde{R}(D)$ and $\Tilde{S}(D)$ respectively, so we get that $S_D$ is row equivalent to $R_D$.
    \end{proof}
    
    We now look at various properties of this parametrization. We start by obtaining a simple product formula for $\Tilde{R}_D$. For $\beta\in\mathbb{F}$, let $X_{(i,j)}(\beta)$ be the matrix which has 1's on the diagonal, and the only non-zero off-diagonal entry is $\beta$, which is in row $i$ and column $j$. Note that for any $D$, $\Tilde{R}(D)$ is going to be a lower triangular matrix, because restricted paths always jump from a higher index pipe to a lower index pipe. This suggests that we should be able to write it as a product of $X_{(i,j)}(\beta)$ where $i>j$. The lemma below describes one such factorization:
    \begin{lemma}\label{prodlemma}
        For a Go-diagram $D$ inside partition $\lambda$, and the set of jump parameters $\{\beta_b\}_{b\in \lambda}$ satisfying the constraints in \cref{resparameterthm}, we have:
        \begin{equation}\label{eqn:productformula}
            \Tilde{R}_D = \prod_{b\in \lambda}X_{\sigma(b)}(\beta_b)
        \end{equation}
        where a reading order decides the ordering in the product on $\lambda$, cells with a higher index appear left of the ones with a lower index in the reading order.
    \end{lemma}
    \begin{proof}
        We shall prove \cref{eqn:productformula} for a particular reading order first. In our reading order, the leftmost column gets the biggest entries, starting from the top, and we move towards the right until we label all the cells. \cref{fig:readingwordprod} gives an example of this for a particular partition. Let $E_{(i,j)}$ denote the matrix of all 0's except at row $i$ and column $j$ where we have 1. Therefore we have $X_{(i,j)}(\beta) = 1+\beta E_{(i,j)}$. For a column $c$ in $\lambda$, we have:
        \begin{equation*}
            X(c) := \prod_{b\in c} X_{\sigma(b)}(\beta_b) = \prod_{b\in c} (1+\beta_b E_{\sigma(b)}) = 1+\sum_{b\in c}\beta_b E_{\sigma(b)}
        \end{equation*}
        as $E_{\sigma(b)}E_{\sigma(b')} = 0$ if $b$ is strictly above $b'$. Moreover, if $b'$ is in a column right next to that of $b$ and both $b$ and $b'$ don't have white stones, then $E_{\sigma(b)}E_{\sigma(b')} \not=0 \iff \prescript{-}{}{b} = \prescript{+}{}{b'}$. Let $C$ be the set of all columns in $\lambda$, and $R$ be the set of all restricted paths in $D$ from sources to sinks. For a restricted path $r\in R$, let $\sigma(r) = (i,j)$ where $i$ and $j$ are the source and sink respectively for $r$. Then we have the following:
        \begin{align*}
            \prod_{b\in \lambda}X_{\sigma(b)}(\beta_b)&= \prod_{c\in C}X(c)\\
            &= \prod_{c\in C}\left(1+\sum_{b\in c}\beta_b E_{\sigma(b)}\right)\\
            &= \sum_{r\in R}\mathrm{wt}_1(r)E_{\sigma(r)}\\
            &= \Tilde{R}_D.
        \end{align*}

        Now we note that any two reading orders are related to each other by a sequence of swaps, where we are swapping entries like $b_2$ and $b_3$ as in \cref{fig:readingwordprod}. In this case no matter what is configuration in each cell, we always have $\{\prescript{-}{}{b_2}, \prescript{+}{}{b_2}\}\cap \{\prescript{-}{}{b_3},\prescript{+}{}{b_3}\} = \emptyset$, hence $X_{\sigma(b_2)}(\beta_{b_2})X_{\sigma(b_3)}(\beta_{b_3})=X_{\sigma(b_3)}(\beta_{b_3})X_{\sigma(b_2)}(\beta_{b_2})$, therefore the product in \cref{eqn:productformula} is invariant of reading order.
    \end{proof}
    \begin{figure}
        \centering
        \includesvg[width=75mm, height=35mm]{readingwordprod} 
        \caption{Example of reading word for \cref{prodlemma} (left), commuting cells in $\lambda$ (right)}
        \label{fig:readingwordprod}
    \end{figure}
    \begin{example}\label{ex:product}
        Let $D$ be the Go-diagram from \cref{ex:respaths}. Here is the table that gives the jump coordinates at each cell:
        \begin{equation*}
            \sigma_D=\begin{array}{|c|c|c|}
            \hline
               (4,3)  & (3,2)&(2,1) \\
            \hline
               (5,4)  & (4,3)&(4,2)\\
            \hline
                (6,5)&(5,3)&(3,4)\\
            \hline
            \end{array}
        \end{equation*}
        Using \cref{prodlemma}, we get:
        \begin{equation*}
        \Tilde{R}_D = \begin{array}{ccc}
               X_{(4,3)}(\beta_9)\cdot  & X_{(3,2)}(\beta_8)\cdot&X_{(2,1)}(\beta_{7})\cdot \\
               X_{(5,4)}(\beta_{6})\cdot  & X_{(4,3)}(\beta_{5})\cdot&X_{(4,2)}(\beta_{4})\cdot\\
                X_{(6,5)}(\beta_{3})\cdot&X_{(5,3)}(\beta_2)\cdot&X_{(3,4)}(0)\\
            \end{array}
    \end{equation*}
    \end{example}
\subsection{Dual restricted paths}
    We have the isomorphism $Gr_{k,n}\simeq Gr_{n-k,n}$ with the bijection given as follows:
    \begin{equation}\label{eq:dual}
        V\in Gr_{k,n} \leftrightarrow \{ f\in (\mathbb{F}^n)^\star|f(v)=0 \mathrm{\;for\;all\;}v\in V\}
    \end{equation}
    where $(\mathbb{F}^n)^\star$ is the dual space of $\mathbb{F}^n$.
    
    Let the image of $V$ under this map be $V^\star$. Also let $D$ be the Go-diagram of shape $\lambda$ such that $V\in \mathcal{D}_D$, let $\{\beta_b\}_{b\in\lambda}$ be the parameters along which realize $V$ as a restricted path matrix on $D$. Finally, let $R_D$ be the restricted path matrix corresponding to $V$. In this subsection, we describe how to obtain $V^\star$ using the parameters $\{\beta_b\}_{b\in\lambda}$. We first "dualize" the restricted paths. We create source and sink vertices on the boundary of the Go-diagram, similar to how we did for \cref{def:respaths}.
    \begin{definition}[Dual restricted paths]
        A \textbf{dual restricted path} is a path between a source and sink vertex, which starts at a pipe and is allowed to jump between pipes at the cells of $\lambda$. The path has to satisfy one of the following rules when it enters a cell which has $+$ or a $\bullet$:
            \begin{enumerate}
                \item The path follows the pipe it was on.
                \item The path jumps to a pipe with the \textbf{higher} index.
            \end{enumerate}
         A \textbf{jump site} for the dual restricted path is a cell where the path changes pipes.
          \cref{fig:dualjumpcases} shows when the possible ways a restricted path can jump (the red coloured pipe has lower index). See \cref{ex:dualrespaths} for some examples of dual restricted paths.
          \begin{figure}[h]
        \centering
        \includesvg[width=90mm, height=30mm]{dualjumpcases}
        \caption{Possible ways for a restricted path to jump.}
        \label{fig:dualjumpcases}
    \end{figure}
    \end{definition}
    \begin{definition}
        We also define the \textbf{dual jump coordinate} $\sigma^\star_D(b)= rev(\sigma_D(b))$ for all cells $b$. Here $rev(a,b):=(b,a)$. We have defined it in this way to encode the jump conditions on the dual restricted paths. We attach the jump parameter $\beta_b^\star$ at cell $b$.
    \end{definition}
    The weight of the dual restricted path is given by the product of the jump parameters $\beta_b^\star$ at the jump sites of the path. 
    \begin{figure}
        \centering
        \includesvg[width=35mm, height=35mm]{rp_1}
        \includesvg[width=35mm, height=35mm]{drp_2}
        \caption{Examples of dual restricted paths}
        \label{fig:dualrespaths}
    \end{figure}
    The weight matrix of this network is:
    \begin{equation*}
        (\Tilde{R}_D^\star)_{st}=\sum_{P:s\to t}\mathrm{wt}(P)
    \end{equation*}
    where the sum runs over all dual restricted paths from source $s$ to sink $t$. Let $v_1,v_2,...,v_{n-k}$ be the labels that appear on the north boundary from \textbf{right to left}. Let $R_D^\star$ be a $k\times n$ matrix defined as follows:
    \begin{equation*}
        (R_D^\star)_{st} = (\Tilde{R}_D^\star)_{v_st}
    \end{equation*}
    \begin{example}\label{ex:dualrespaths}
        Let $D$ be the Go-diagram in \cref{fig:resnetwork} and the reading order on cells is given by \cref{fig:readingword1}. \cref{fig:dualrespaths} gives some examples of dual restricted paths on $D$. The one on the left is a pipe, so it has weight 1, the one on the right has weight $\beta_9^\star$. In this case,
        \begin{align*}
            \Tilde{R}_D^\star &= \left[
            \begin{matrix}
                1&\beta_7^\star& 0 &\beta_7^\star\beta_4^\star&0&0\\
                0&1&\beta_8^\star&\beta_8^\star\beta_5^\star+\beta_4^\star&\beta_8^\star\beta_2^\star&0\\
                0&0&1&\beta_9^\star+\beta_5^\star&\beta_9^\star\beta_6^\star+\beta_2^\star&\beta_9^\star\beta_6^\star\beta_3^\star\\
                0&0&0&1&\beta_6^\star&\beta_6^\star\beta_3^\star\\
                0&0&0&0&1&\beta_6^\star\\
                0&0&0&0&0&1
            \end{matrix}\right]\\
            R_D^\star&=\left[
            \begin{matrix}
                1&\beta_7^\star& 0 &\beta_7^\star\beta_4^\star&0&0\\
                0&1&\beta_8^\star&\beta_8^\star\beta_5^\star+\beta_4^\star&\beta_8^\star\beta_2^\star&0\\
                0&0&1&\beta_9^\star+\beta_5^\star&\beta_9^\star\beta_6^\star+\beta_2^\star&\beta_9^\star\beta_6^\star\beta_3^\star\\
            \end{matrix}\right]\\
        \end{align*}
        
    \end{example}
    Similar to how we had a product description for $\Tilde{R}_D$, we also have a product description for $\Tilde{R}_D^\star$, and the proof for it is similar to the proof of \cref{prodlemma}:
    \begin{lemma}\label{dualprodlemma}
        For a Go-diagram $D$ inside partition $\lambda$, and the set of jump parameters $\{\beta_b^\star\}_{b\in \lambda}$ satisfying the constraints in Theorem \ref{resparameterthm}, we have:
        \begin{equation}\label{eqn:dualproductformula}
            \Tilde{R}_D^\star = \prod_{b\in \lambda}X_{\sigma_D^\star(b)}(\beta_b^\star)
        \end{equation}
        where a reading order decides the ordering in the product on $\lambda$, cells with a higher index appear left of the ones with a lower index in the reading order.
    \end{lemma}
    Now we can show that we can realize $V^\star$ using $R_D^\star$.
    \begin{theorem}
        Let $V\in Gr_{k,n}$, $D$ be the Go-diagram such that $\mathcal{D}_D$, $\{\beta_b\}_{b\in \lambda}$ be the parameters for the restricted path matrix $R_D$ which realizes $V$ in the standard basis for $\mathbb{F}^n$. Then if we set $\beta_b^\star= -\beta_b$ for all $b\in \lambda$, then $R_D^\star$ realizes $V^\star$ in the basis dual to the standard basis.
    \end{theorem}
    \begin{proof}
        The above theorem follows directly from \cref{prodlemma} and \cref{dualprodlemma} which allow us to see that if $\beta_b^\star = -\beta_b$ for all $b\in \lambda$ we have:
        \begin{equation*}
            \Tilde{R}_D(\Tilde{R}_D^\star)^T = I
        \end{equation*}
        where $A^T$ denotes the transpose of matrix $A$.
    \end{proof}

\subsection{Pl\"{u}cker coordinates}
    Given a point $V\in \mathcal{D}_D$ where $D$ is a Go-diagram, we can write the Pl\"{u}cker coordinates using the Lindstr\"{o}m-Gessel-Viennot lemma on the Talaska-Williams network or the restricted path network. However, the summands appear to have no other structure beyond originating from non-intersecting paths on either network. Furthermore, if we want to check if a particular Pl\"ucker coordinate is nonzero, the only option we have is to guess a nonintersecting system of paths in either network. In this section, we will use the restricted path parametrization to give additional structure to the summands. In particular, we will be obtaining all the possible summands by performing a sequence of moves on the Go-diagram $D$, which constructs a graph out of the summands. This picture has an additional advantage, we also generate the summands in the Pl\"{u}cker coordinates for $V^\star$ as well. Furthermore, once we have constructed the graph obtained using these moves, it is straight forward to see if a Pl\"ucker coordinate is 0, we just need to see if our graph has vertices corresponding to the coordinate we are interested in. The key feature that we use in this section is that at a jump site, restricted paths (or dual restricted paths) behave as if the filling at that cell has been "toggled", i.e., a $+$ has been swapped with $\circ$ or vice versa. We first give an example to illustrate this idea.
    \begin{example}
        Let $D$ be the Go-diagram in \cref{fig:resnetwork}, $I=\{3,5,6\}$. The non-intersecting system of restricted paths for $\Delta_{356}(R_D)$ is shown in \cref{fig:nrp}. We "toggle" the jump sites of the restricted paths to obtain $(+,\circ)$ filling as in \cref{fig:rd}. Finally, looking at the pipes of these diagrams coming from the north boundary, we get the non-intersecting system of dual restricted paths for $\Delta_{124}(R_D^\star)$ in \cref{fig:dnrp}.
    \begin{figure}
            \centering
            \includesvg[width =82mm, height=50mm]{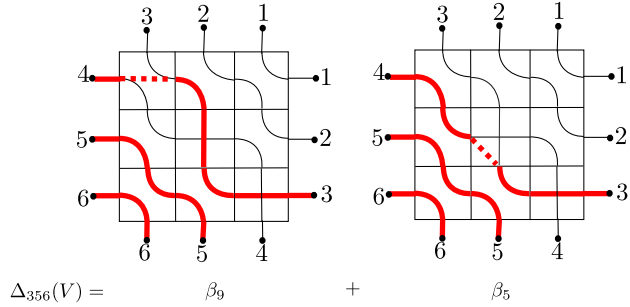}
            \caption{Non-intersecting system of restricted paths for $\Delta_{356}(R_D)$}
            \label{fig:nrp}
    \end{figure}
    \begin{figure}
            \centering
            \includesvg[width =82mm, height=50mm]{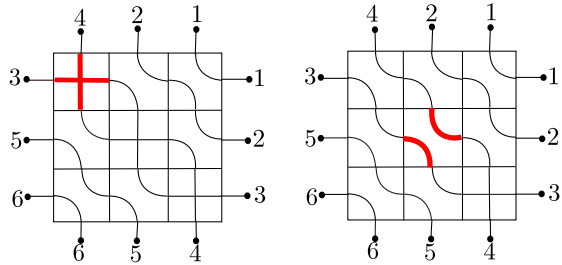}
            \caption{Restricted diagrams for $\Delta_{356}(R_D)$}
            \label{fig:rd}
    \end{figure}
    \begin{figure}
            \centering
            \includesvg[width =82mm, height=50mm]{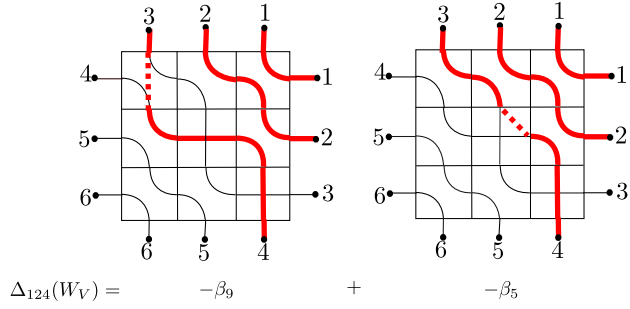}
            \caption{Non-intersecting system of dual restricted paths for $\Delta_{124}(R_D^\star)$}
            \label{fig:dnrp}
    \end{figure}
    \end{example}
    We first define the toggling move that we use in the above example.
     \begin{definition}[Togglable cell and Toggled cells] 
        We define this recursively. For a $(+, \circ)$ diagram $E$ with some toggled cells, we define a togglable cell $p$ as follows:
        \begin{enumerate}
    	\item Toggling the cell (i.e. putting $\circ$ in the cell instead of $+$ if $D$ had a $+$ there and vice versa) swaps $b$ and $c$ where $b$ is the label of pipe on a west boundary, and $c$ is a label on the north boundary.
    	\item Toggling the cell increases the length of the permutation in the diagram
    	\item There is no toggled cell $q$ that is also inside $p_{in}$
        \end{enumerate}
        Let $E^p$ be the diagram obtained from $E$ after toggling $p$. The toggled cells of $E^p$ are all the cells which were toggled in $E$ along with $p$.
    \end{definition}

    \begin{definition}[Restricted diagram]
        We start with $E_0=D$ and no cells are toggled. Here we will impose an additional restriction on the togglability of cells: all cells with a white stone in $D$ are \textbf{untogglable}. We will interpret $D$ as $(+,\circ)$ filling by decolouring the black stones. A diagram $E$ is a \textbf{restricted diagram} obtained from $D$ if there is a sequence $D=E_0, E_1, E_2,...,E_k=E$ where $E_{i+1}$ is obtained by toggling a togglable cell in $E_i$. 
    \end{definition}
    Let $\mathfrak{R}(D)$ (resp. $\mathfrak{R}^\star(D)$) be the set of all non-intersecting systems of restricted paths (resp. dual restricted paths) on $D$. Let $\mathfrak{D}(D)$ be the set of all restricted diagrams obtained from $D$. The main theorem of this subsection is:
    \begin{theorem}\label{thm:resdiagram}
        Let $D$ be a Go-diagram. We have bijections $f_1:\mathfrak{D}(D)\to \mathfrak{R}(D), f_2:\mathfrak{D}(D)\to \mathfrak{R}^\star(D)$, where for a restricted diagram $E$, $f_1(E)$ (resp. $f_2(E)$) are the pipes coming from the left (resp. north) boundary of the pipedream corresponding to $E$.
    \end{theorem}
    The proof of the above theorem obtained by recognizing the fact that the jump sites of the restricted paths in a non-intersecting system correspond to the toggled cells in a restricted diagram $E$ (that is why we declared the cells with white stones in $D$ as untogglable as these cells can never be jump sites). Before proving the theorem, we show the following lemma:
    \begin{lemma}\label{lemma:resconfig}
        If $E$ is a restricted diagram obtained from $D$, and $p$ is a cell in $E$, with pipes coming to this cell labelled $b$(this pipe ends at the west boundary) and $c$(this pipe ends at the north boundary). Then condition (2) for togglable cells is equivalent to $c<b$. Moreover, if $p$ is togglable and $E\in \mathfrak{D}(D)$ and $p$ does not have a white stone in $D$, then $b$ is exiting $p$ from the west boundary of the cell, and $c$ is exiting the north boundary of the cell.
    \end{lemma}
    \begin{proof}
        Toggling cell $p$ is equivalent to left multiplying the permutation corresponding to the pipedream of $E$ by the transposition swapping $b$ and $c$, which increases the inversion number of the permutation if and only if $c<b$. For the second half of the lemma, we proceed by induction on the number of toggled cells in $E$. For $E=D$, since $E$ is a Go diagram, and $p$ is not a white stone, the pipe exiting the west boundary of $p$ will have a smaller label. Now let $E$ have $k>0$ toggled cells. Therefore we have diagrams $D=E_0, E_1,..., E_k=E$ such that $E_{i+1}$ is obtained by toggling a togglable cell in $E_i$. Since $p$ is a togglable cell in $E$, no toggled cell is inside $p_{in}$. Therefore $(+,\bigcirc)$ arrangement inside $p_{in}$ is same for both $E$ and $D$. Therefore, the pipes exiting the west and east boundary of the $p$ in $E$ are the same as the ones in $D$ that is $b$ and $c$ respectively.
    \end{proof}
    \begin{proof}[Proof of \cref{thm:resdiagram}]
        Let $E$ be a reduced diagram obtained from $D$. We show that the pipes ending on the west boundary of the pipe dream corresponding to $E$ are a system of non-intersecting restricted paths (let's call them $P$), and the pipes ending at the north boundary of the pipe dream corresponding to $E$ are a system of non-intersecting dual restricted paths (lets call them $Q$). We shall show this by induction on the number of toggled cells in $E$. For $E=D$, the statement is obvious. So let $k>0$ cells have been toggled. Therefore we have diagrams $D=E_0, E_1,..., E_k=E$ such that $E_{i+1}$ is obtained by toggling a togglable cell in $E_i$. Let $p$ be the cell that was toggled in $E_{k-1}$ to obtain $E$. Let $b$ and $c$ be the labels of the pipes entering $p$ with $b>c$. From the \cref{lemma:resconfig} and the definition of a togglable cell, $b$ ends up at the west boundary, $c$ ends up at the north boundary in $E_{k-1}$. By induction, the west ending pipes of $E_{k-1}$ represent a non-intersecting restricted path system $P'$, and the north ending pipes represent a non-intersecting dual restricted path system $Q'$.
        
        Since $E\in \mathfrak{D}(D)$, we know that $D$ didn't have a white stone at $p$. Now, from the previous lemma, we have that the pipe labelled $b$, which is a restricted path in $P'$ (let this be $r_W$), is entering $p$ from the west boundary of the cell, while the pipe labelled $c$ is a dual restricted path in $Q'$ (let this be $r_N$) entering from the north boundary of the cell. Now let $r_W$ be subdivided as $r_W''pr_W'$ where $r_W''$ is part of the pipe $r_W$ between $p$ and the west boundary, and $r_W'$ is the part that comes between $p$ and the south-east boundary. Similarly, split the dual restricted path $r_N$ as $r_N=r_N''pr_N'$. Since $p$ was togglable in $E_{k-1}$, there is no toggled cell in $p_{in}$. So the $(+,\bigcirc)$ arrangement in $p_{in}$ is same for $E_{k-1}$ and $D$. Since $r_W',r_N'$ lie entirely in $p_{in}$, they are segments of pipe in $D$, and hence they satisfy the conditions for both restricted paths and dual restricted paths. Hence, the path $s_W=r_W''pr_N'$ is a restricted path, and $s_N=r_N''pr_W'$ is a dual restricted path (since $p$ didn't have a white stone, restricted and dual restricted paths can jump at $p$). Finally, note that $P$ can be obtained from $P'$ by removing $r_W$ as a restricted path and replacing it with $s_W$ (similarly, we replace $r_N$ with $s_N$ in $Q$). So $P$ is a system of restricted paths, $Q$ is a system of restricted paths, and they are non-intersecting as they are built off the pipes of $E$.
        
        Now we give the inverse of the above map, that is, given $P\in \mathfrak{R}(D)$ then we give $E\in \mathfrak{D}(D)$ with $f_1(E) =P$. We define $E$ as the diagram obtained by toggling all the cells in $D$ which are jump sites in the restricted paths of $P$. This rule is consistent as two restricted paths enter the same cell in a non-intersecting way if and only if neither of them jumps at this common cell, so no two restricted paths in $P$ share a jump site. Moreover, none of the cells with white stones in $D$ are going to be toggled as restricted paths are not allowed to jump at white stones. It is easy to see if $E\in \mathfrak{D}(D)$, then $f_1(E) = P$. To show $E\in \mathfrak{D}(D)$, we just need to produce a sequence of cells that, on toggling in $D$ to obtain $E$.
        
        We shall proceed by induction on the number of jump sites of restricted paths in $P$. If there were no jump sites, then we get $E=D\in \mathfrak{D}(D)$. Let the number of jump sites be $k>0$. Let $p$ be a jump site in $P$ such that there are no jump sites $p_{in}$. Now we can swap the parts between $p$ and the south-east boundary just as in the previous part of the proof, to obtain $P'\in \mathfrak{R}_D$ where $p$ is no longer a jump site. Using induction, we get a restricted diagram $E' \in \mathfrak{D}(D)$. Since at $p$ the restricted path in $P'$ had an option to jump without intersecting with other paths in the system, there is no other restricted path entering $p$ in $P'$. Hence, the other pipe entering $p$ in $E'$ corresponds to a dual restricted path, showing that $p$ satisfies condition(1) for togglability. Since $P$ is obtained from $P'$ by introducing a jump site, and since restricted paths only move to a pipe of higher label after toggling, we see that by \cref{lemma:resconfig} $p$ satisfies condition (2) of togglablity. Finally condition (3) is ensured by our choice of $p$. Hence $p$ is togglable in $E'$ and $E$ can obtained from $E'$ by toggling $p$, giving us that $E\in\mathfrak{D}(D)$. Inverse of $f_2$ is defined similarly.
    \end{proof}
    \begin{figure}
            \centering
            \includesvg[width =120mm, height=100mm]{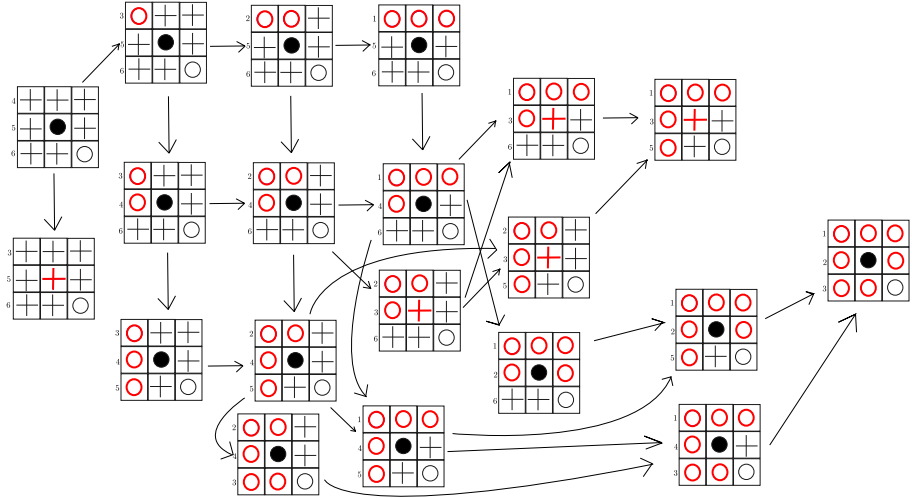}\caption{Graph generated by toggling moves on the Go-diagram $D$}
            \label{fig:toggling}
    \end{figure}
    \begin{example}
        Let $D$ be the Go-diagram from \cref{ex:respaths}. \cref{fig:toggling} shows the graph $\mathfrak{D}(D)$ generated by our toggling operations. The toggled entries are coloured in red.
    \end{example}
    For a restricted diagram $E\in \mathfrak{D}(D)$, let $(p_1,p_2,...,p_k)$ be the labels of the pipe on the west boundary read from top to bottom. Let $I_E=\{p_1,p_2,...,p_k\}$ and $\pi_E$ be the permutation such that $p_{\pi_{E}(1)}\leq p_{\pi_{E}(2)}\leq ....\leq p_{\pi_{E}(k)}$. Then we can write the Pl\"{u}cker coordinates as a sum of restricted diagrams:
    \begin{theorem}
        Let $D$ be Go-diagram in partition $\lambda\in \Lambda_{k,n}$ and $I\in \binom{[n]}{k}$. Then we have:
        \begin{equation*}
        \Delta_I(R_D) = \sum_{\substack{E\in\mathfrak{D}(D)\\I_E=I}}\mathrm{sign}(\pi_E)\mathrm{wt}(E)
        .
    \end{equation*}
    where $\mathrm{wt}(E)$ is the product of jump parameters on the toggled sites of $E$
    \end{theorem}
    We end the section with the following question:
    \begin{question}
        Given a Go-diagram $D$ inside a partition $\lambda$, and $E$ which is $(+,\circ)$ filling of $\lambda$, give a condition (necessary or sufficient) for when $E\in \mathfrak{D}(D)$, which doesn't require us to construct the entire graph on $\mathfrak{D}(D)$.
    \end{question}
    
\section{Closure of Deodhar components}\label{sec:closure}
In this section, we will investigate the closure relations between Deodhar components. For a Go-diagram $D$, denote $\pi(D)$ as the permutation associated with the pipedream corresponding to $D$. We first focus on the case when the permutations corresponding to the pipedreams of our Go-diagrams are the identity permutation: 
    \begin{theorem}\label{theorem:closurerelation}
         Let $D'$ be a Go-diagram, and let $c$ and $c'$ be two cells in $D'$ which form a crossing and uncrossing pair in $D'$; $c'$ is the cell with $\bullet$. Let $D$ be the diagram obtained after changing the stones at $c$ and $c'$ to $+$. Suppose $i,i+1$ are the two pipes which form the crossing at $c$ and $c'$ and also assume that $\pi(D')=1$.
        In this case we have that $D$ is a Go-diagram and $\mathcal{D}_{D'}\subset \overline{\mathcal{D}_D}$.
    \end{theorem}
    The fact that $D$ is a Go-diagram is easy to see, as if $D$ has a forbidden configuration, it involves a pipe with label $k\not=i,i+1$, so this will form a forbidden configuration even in $D'$. In \cref{sec:generalclosure}, we will work towards removing the requirement. The key lemma for proving \cref{theorem:closurerelation} is \cref{prodlemma}. We use \cref{prodlemma} to write both $\Tilde{R}_D$, and $\Tilde{R}_{D'}$ both equipped with compatible set of parameters (See \cref{resparameterthm}). Then we describe a $D'$-distortion of $\Tilde{R}_D$ which transforms the product obtained for $\Tilde{R}_D$ using \cref{prodlemma} to make the product similar to the one for $\Tilde{R}_{D'}$. The upshot of this is that we will easily be able to obtain a sequence of points, so that $\lim \Tilde{R}_D = \Tilde{R}_{D'}$, where in the limit, one of the parameters for $\Tilde{R}_D$ approaches infinity. Finally, truncating the matrices in the above sequence, we will get $\lim R_D = R_{D'}$, proving the theorem.

    Before proving the theorem, we define $W_{(i,j)}(\beta):=X_{(i,j)}(\beta)X_{(j,i)}(-\frac{1}{\beta})X_{(i,j)}(\beta)$ for all $\beta\in \mathbb{F}^*$. These are parametrized versions of $\Dot{s}_i$ that appear in the Marsh and Rietsch parametrization when $j=i+1$. We also note the following identities, which can be easily verified by restricting ourselves to $3\times 3$ matrices:
    \begin{lemma}\label{identites}
        Let $\beta\in \mathbb{F}^*, \gamma\in \mathbb{F}, i,j,k\leq n, i\not=j\not=k$. 
        \begin{enumerate}
            \item $W_{(i,j)}(\beta)W_{(i,j)}(-\beta) = 1$
            \item $W_{(i,j)}(-\beta)X_{(j,k)}(\gamma)W_{(i,j)}(\beta) = X_{(i,k)}(-\beta\gamma)$
            \item $W_{(i,j)}(-\beta)X_{(i,k)}(\gamma)W_{(i,j)}(\beta) = X_{(j,k)}(\frac{\gamma}{\beta})$
            \item $W_{(i,j)}(-\beta)X_{(k,i)}(\gamma)W_{(i,j)}(\beta) = X_{(i,k)}(\beta\gamma)$
            \item $W_{(i,j)}(-\beta)X_{(k,j)}(\gamma)W_{(i,j)}(\beta) = X_{(i,k)}(-\frac{\gamma}{\beta})$
            \item $X_{(i,j)}(-\gamma)X_{(j,k)}(-\beta)X_{(i,j)}(\gamma)X_{(j,k)}(\beta)= X_{(i,k)}(\beta\gamma)$
            \item $X_{(i,j)}(-\gamma)X_{(k,i)}(-\beta)X_{(i,j)}(\gamma)X_{(k,i)}(\beta)= X_{(k,j)}(-\beta\gamma)$
        \end{enumerate}
    \end{lemma}
    Note that the action of conjugating by $W_{(i,j)}(\beta)$ on $X_p(\gamma)$ changes the subscript $p$ to $t_{i,j}(p)$ where $p$ is some ordered pair not equal to $(i,j)$ or $(j,i)$ and $t_{i,j}$ is the transposition swapping $i$ and $j$. For the proof of \cref{theorem:closurerelation} we take the matrix $\Tilde{R}_D$ and conjugate all entries between $c$ and $c'$ with $W_{(i+1,i)}$ to make the subscripts showing up in the product match with the ones showing up in the product for $\Tilde{R}_{D'}$. Unfortunately, the conjugation itself produces extra terms; the lemma below provides a tool to fix this problem. 
    \begin{lemma}\label{blackstonelemma}
        Let $D'$ be a Go-diagram with a crossing and an uncrossing pair at $c$ and $c'$, and the crossing is between pipes $i$ and $j$ where $i<j$. Let us also assume that $\pi(D') = 1$. Then the following is true about $D'$:
        \begin{itemize}
            \item If there is cell $c\prec b\prec c'$ with $\sigma_{D'}(b)=(k,j)$ or $(j,k)$ in $D'$, between $c$ and $c'$, $k>i$, then there is unique minimal cell $d'$ with a black stone such that $\sigma_{D'}(d')=(k,i)$. If $d$ is the cell with the corresponding white stone, then we have $d\prec b \prec d'$.
            \item If there is cell $c\prec b\prec c'$ with $\sigma_{D'}(b)=(i,k)$ or $(k,i)$ in $D'$, between $c$ and $c'$, $j>k$, then there is unique minimal cell $d'$ with a black stone such that $\sigma_{D'}(b')=(j,k)$. If $d$ is the cell with the corresponding white stone, then we have $d\prec b \prec d'$.
        \end{itemize}
    \end{lemma}
    \begin{proof}
        We will prove the first part; the proof for the second part is similar, just with the inequalities reversed. Choose an arbitrary reading order on $D'$, and let $l$ be the number at cell $b$. Let $\textbf{v}$ be the subexpression of identity corresponding to $D'$. Since $b$ is between $c$ and $c'$, in $v_{(l)}$, $(i,j)$ is an inversion pair. Moreover since $\sigma_{D'}(b) =(k,j)$ or $(j,k)$, in $v_{(l)}$, $k$ and $j$ are adjacent to each other. This means that even $(i,k)$ is an inversion pair in $v_{(l)}$. Hence, these two pipes have crossed each other before $b$ and since $\pi(D') = 1$, they will uncross each other eventually. Note that the cells traversed by any pipe are totally ordered by $\prec$. So we get that $d'$ is simply the minimal cell where they uncross after $b$, $d$ is the maximal cell where they cross each other before $b$.
    \end{proof}
    Now we are going to describe the $D'$-\textbf{distortion} of $\Tilde{R}_D$. Fix a reading order and arrange the cells of $\lambda$ in the linear order using the chosen reading order. Now, the product formula of $\Tilde{R}_D$ from \cref{prodlemma} assigns a factor to each cell. We will fill each cell of $\lambda$ with factors coming from \cref{prodlemma}. Let $\gamma_b$ be the parameter at cell $b$. We conjugate every entry between $c$ and $c'$ with $W_{i+1,i}(\gamma_c)$:
    \begin{align*}
        \Tilde{R}_D &= ...|X_{(i+1,i)}(\gamma_{c'})|...|X_{\sigma_D(b)}(\gamma_b)|...|X_{(i+1,i)}(\gamma_c)|...\\
        &= ...|X_{(i+1,i)}(\gamma_{c'})W_{(i+1,i)}(\gamma_c)|...|W_{(i+1,i)}(-\gamma_c)X_{\sigma_D(b)}(\gamma_b)W_{(i+1,i)}(\gamma_c)|...|W_{(i+1,i)}(-\gamma_c)X_{(i+1,i)}(\gamma_c)|...\\
        &= ...|X_{(i+1,i)}(\gamma_{c'}+\gamma_c)X_{(i,i+1)}(-\frac{1}{\gamma_c})X_{(i+1,i)}(\gamma_c)|...|X_{\sigma_{D'}(b)}(\gamma_b')|...|X_{(i+1,i)}(-\gamma_c)X_{(i,i+1)}(\frac{1}{\gamma_c})|...
    \end{align*}
    Here, we have used $|$ to seperate out the factors in different cells of $\lambda$. Now given $\Tilde{R}_{D'}$ with some choice of parameters $\beta_{\sigma_{D'}(b)}$ we would like $\Tilde{R}_D$ to converge to $\Tilde{R}_{D'}$ when $\beta_b\to\infty$ while fixing the parameters in other factors. Since now the subscripts of factors in all cells except $c$ and $c'$ match with the subscripts of the factors in $D'$, we would like to keep the content in these cells constant, while sending $\gamma_c\to \infty$. If we send $\gamma_c\to\infty$ right now, the factors in cells $c$ and $c'$ will blow up due to the presence of $X_{(i+1,i)}(-\gamma_c)$ and $X_{(i+1,i)}(\gamma_c)$ at $c$ and $c'$ respectively.  So, we would like these factors to cancel each other. Hence, we will move the factor $X_{(i+1,i)}(-\gamma_c)$ all the way to $c'$. While moving this to the left, when we pass through cell $b$, we encounter the following cases:
    \begin{itemize}
        \item Case 1: $b$ has white stone in $D$. Since $c\prec b\prec c'$, and pipes at $c$ and $c'$ are $i$ and $i+1$ (adjacent), we get that $b$ has a white stone in $D'$ as well. Therefore the factor at cell $b$ is identity matrix, and therefore $X_{(i+1,i)}(-\gamma_c)$ passes through $b$ freely without changing anything at that cell.
        \item Case 2: $\sigma_D(b)=(k,l)$, $\{k,l\}\cap\{i,i+1\}=\emptyset$. In this case, $\gamma_b' = \gamma_b, \sigma_D(b) = \sigma_{D'}(b)$ and $X_{(i+1,i)}(-\gamma_c)$ commutes with $X_{\sigma(b)}(\gamma_b)$, so we can freely move $X_{(i+1,i)}(-\gamma_c)$ to the left. So the factor at cell $b$ is unaffected.
        \item Case 3: $\sigma_D(b) = (i,k)$ or $\sigma_D(b) = (k,i+1)$. If $\sigma_D(b) = (i,k)$ (resp. $(k,i+1)$) $b$ doesn't have a white stone in $D$ (and $D'$), we have $k<i$ (resp. $i+1<k$), $\sigma_{D'}(b) = (i+1,k)$ (resp. $(k,i)$), and $\gamma_b' = -\gamma_b\gamma_c$ (resp. $\gamma_b\gamma_c$). In this case $X_{(i+1,i)}(-\gamma_c)$ will commute with $X_{\sigma_{D'}(b)}(\gamma_b')$, hence again, we can freely move $X_{i+1,i}(-\gamma_c)$ to the left and the entry at the cell $b$ is unaffected.
        \item Case 4: $\sigma_D(b) = (k,i)$ or $(i+1,k)$. Let us assume $\sigma_D(b) = (k,i)$, the other one is handled similarly. Since $b$ doesn't have a white stone in $D$ (and $D'$), we have $k>i$. We have $\gamma_b' = -\frac{\gamma_b}{\gamma_c}$ and $\sigma_{D'}(b)=(k,i+1)$. This time $X_{(i+1,i)}$ does \textbf{not} commute with $X_{\sigma_{D'}(b)}(\gamma_b')$. Therefore, we use the following identity (identity 7 in \cref{identites}):
        \begin{equation*}
            X_{(k,i+1)}\left(-\frac{\gamma_b}{\gamma_c}\right)X_{(i+1,i)}(-\gamma_c) = X_{(k,i)}(\gamma_b)X_{(i+1,i)}(-\gamma_c)X_{(k,i+1)}\left(-\frac{\gamma_b}{\gamma_c}\right).
        \end{equation*}
        This allows us to move $X_{(i+1,i)}(\gamma_c)$ to the left at the cost of producing an extra factor $X_{(k,i)}(\gamma_b)$ at cell $b$. 
    \end{itemize}
    After doing the above movement, we got rid of the factor $X_{(i+1,i)}(\gamma_c)$, with two factors at cell $c'$ ($X_{(i+1,i)}(\gamma_c+\gamma_{c'})$ and $X_{i,i+1}(\frac{1}{\gamma_c})$), but now we have produced extra factors at some cells whenever Case 4 was triggered.
    \begin{definition}[Excited factor]
        We say a factor at cell $b$ is \textbf{excited}, if:
        \begin{itemize}
            \item $b\not =c'$, and subscript of the factor is not the same as $\sigma_{D'}$.
            \item $b= c'$, and the subscript of the factor is not $(i+1,i)$ or $(i+1,i)$
        \end{itemize}
    \end{definition}
    Whenever Case 4, is triggered, we produce an excited entry at the cell $b$. Case 4 is triggered at every cell $b$ with jump coordinates $\sigma_{D'}(b) = (k,i+1)$ or $(i,k)$ (with $k>i$ in the first case, and $k<i$ in other case as as $b$ doesn't have a white stone in $D'$). For each such cell $b$, using \cref{blackstonelemma}, we get a cell $d'$ with a black stone in $D'$ which has the same jump coordinate as the subscript of the excited entry. We will call $d$ the \textbf{cooling} site for the excited factor at $b$. Now starting with the leftmost excited factor, we move the excited factors leftward to their respective cooling sites. The case analysis of this movement is similar to that of $X_{(i+1,i)}(\gamma_c)$ (in this language, we think of $X_{(i+1,i)}(-\gamma_c)$ as the first excited factor). Therefore, in this movement, we may produce new excited entries; however, whenever Case 4 is triggered, we can also use \cref{blackstonelemma} to find a cooling site for the new excited entries. So we keep moving excited entries leftward to their cooling sites until we no longer have excited entries. This process changes the factors in each cell, and this arrangement of factors in each cell at the end of this process is what we call the \textbf{$D'$-distortion} of $\Tilde{R}_D$. See \cref{ex:distort}. We first show that this process is well defined.
    
    \begin{lemma}\label{lem:distort}
        Let $D', D$, $c$, $c'$ and $i$ be as in \cref{theorem:closurerelation}.Then the process of obtaining $D'$ distortion of $\Tilde{R}_D$ is well defined, that is, we have the following properties:
        \begin{itemize}
            \item If $X_{(k_1,k_2)}(\gamma)$ is an excited factor at any cell, then $k_1>k_2$.
            \item Each newly produced excited factor, while we are moving an excited factor to its cooling site, has a cooling site (so we will be able to move these factors to their cooling sites and continue the process).
            \item The process terminates.
        \end{itemize}
    \end{lemma}
    \begin{proof}
        Let $X_{(k_1,k_2)}(\gamma)$ be an excited factor at cell $b$. To show the first statement, we show that if $k_1>k_2$, then while moving $X_{(k_1,k_2)}(\gamma)$ to its cooling site, the newly excited factors will satisfy the same property. Since the first excited entry is $X_{(i+1,i)}(-\gamma_c)$, we will get the first statement. So assume $k_1>k_2$.  Let $b'\succ b$ be a cell which triggers case 4. We have two cases:
        \begin{itemize}
            \item $b' \not= c'$: Since we are always moving the leftmost excited factor to its cooling site, the fact that case 4 is triggered means that there is exactly one nontrivial factor at cell $b'$ (which is not an excited factor). Therefore, the factor at $b'$ has the same subscripts as $\sigma_{D'}(b')$. Since we triggered case 4, the factor at cell $b'$ is not identity. So $D'$ doesn't have a white stone at cell $b'$; hence, the jump coordinates at cell $b'$ are  $(k_3,k_1)$ with $k_3>k_1>k_2$ or $(k_2,k_3)$ with $k_3<k_2<k_1$. Moreover, our excited factor is moving to its cooling site, so we know that $b'$ lies between a $(k_1,k_2)$ crossing-uncrossing pair in $D'$; hence, we can use \cref{blackstonelemma} to produce a cooling site for our new excited factor. Moreover, our new excited factor has subscripts $(k_3,k_2)$ if $k_3>k_1>k_2$ or $(k_1,k_3)$ if $k_3<k_2<k_1$, so our new factor also satisfies the first statement.
            \item $b' = c'$. Note that since $c'$ is not the cooling site for our excited factor, and since case 4 is triggered at $c'$, we either have $k_1 =i+1$ or $k_2 = i$, but not both. First, let us handle the case when $k_1 = i+1$. Therefore $k_2\not= i$, and since $k_1>k_2$, we get $i>k_2$. Since we are moving the leftmost excited factor, at cell $c'$, we have two factors, and $X_{(i+1,k_2)}(\gamma)$ commutes with $X_{(i+1,i)}(\gamma_c+\gamma_{c'})$ but not with $X_{(i,i+1)}(-\frac{1}{\gamma_c})$. Therefore, when we try moving our excited factor past $c'$, we get a new excited factor $X_{(i,k_2)}(-\frac{\gamma}{\gamma_c})$. Now, since our excited factor is moving to its cooling site, we see that $c'$ lies between a $(i+1,k_2)$ crossing in $D'$. Therefore, we are in the first case of \cref{blackstonelemma}, and we get a cooling site for $X_{(i,k_2)}(-\frac{\gamma}{\gamma_c})$. Note that the new excited factor satisfies the first property. The case where $k_2= i$ is handled similarly.
        \end{itemize} 
        Now we show that the process terminates. When we are moving the excited factor $X_{(k_1,k_2)}(\gamma)$ to its cooling site, all the new excited factors are produced in the cells strictly to the left of $b$ in our reading order if $b\not=c'$. If $b =c'$, and if $k_1 = i$, then we get at most one more excited factor at $c'$ itself. Furthermore, any excited factor produced at the second cell from the left in our reading order will not create any new excited factor. Therefore if $b'$ is the $k^{th}$ cell from the left, in the process of moving $X_{(k_1,k_2)}(\gamma)$ to its cooling site, and then moving any newly produced excited factors produced to the left of $b'$ to their cooling sites, the number of excited factors created is bounded by $a_k$, where the sequence $a_k$ is described by the recurrence relation:
        \begin{align*}
            a_1&=0\\
            a_{k}&= k+\sum_{i=1}^{k-1}a_i \;\;\forall k>2.
        \end{align*}
        In particular, moving $X_{(i+1,i)}(-\gamma_c)$ to its cooling site at $c'$ will produce finitely many excited factors, and thus the process terminates.
    \end{proof}
    
    \begin{example}\label{ex:distort}
        
        \begin{figure}[h]
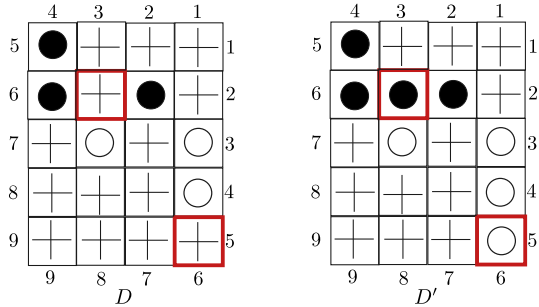

            \centering
            \includesvg[width=30mm, height=40mm]{D1}
            \hspace{2em}
            \includesvg[width=30mm, height=40mm]{D2}
            \caption{Go-diagrams $D$ and $D'$}
            \label{fig:distortion}
        \end{figure}
        Let $D,D'$ be given as in \cref{fig:distortion}. Here $i=5$, and the cells $c,c'$ are highlighted in red. We use a reading order similar to \cref{fig:readingword1}. The jump coordinates for $D$ and $D'$ are as follows:
        \begin{equation*}
        \begin{array}{cc}
            \sigma_D=\begin{array}{|c|c|c|c|}
            \hline
               (5,4)  & (5,3)&(3,2)& (2,1) \\
            \hline
               (6,4)  & \textcolor{red}{(6,5)}&(5,3)&(5,2)\\
            \hline
                (7,4)&(4,6)&(4,3)&(3,5)\\
            \hline
                (8,7)&(7,6)&(6,4)&(4,5)\\
            \hline
                (9,8)&(8,7)&(7,6)&\textcolor{red}{(6,5)}\\
                \hline
            \end{array} &
            \sigma_{D'}=\begin{array}{|c|c|c|c|}
            \hline
               (5,4)  & (5,3)&(3,2)& (2,1) \\
            \hline
               (6,4)  & \textcolor{red}{(6,5)}&(6,3)&(6,2)\\
            \hline
                (7,4)&(4,5)&(4,3)&(3,6)\\
            \hline
                (8,7)&(7,5)&(5,4)&(4,6)\\
            \hline
                (9,8)&(8,7)&(7,5)&\textcolor{red}{(5,6)}\\
                \hline
            \end{array}
        \end{array}
    \end{equation*}
        We write the product in tabular fashion to illustrate how we are distributing the factors at each cell. The factors at cells $c$ and $c'$ are highlighted in red.
    \begin{equation*}
        \Tilde{R}_D = \begin{array}{|c|c|c|c|}
        \hline
               X_{(5,4)}(\gamma_{20})\cdot  & X_{(5,3)}(\gamma_{19})\cdot&X_{(3,2)}(\gamma_{18})\cdot& X_{(2,1)}(\gamma_{17})\cdot \\\hline
               X_{(6,4)}(\gamma_{16})\cdot  & \textcolor{red}{X_{(6,5)}(\gamma_{15})}\cdot&X_{(5,3)}(\gamma_{14})\cdot&X_{(5,2)}(\gamma_{13})\cdot\\\hline
                X_{(7,4)}(\gamma_{12})\cdot&X_{(4,6)}(0)\cdot&X_{(4,3)}(\gamma_{10})\cdot&X_{(3,5)}(0)\cdot\\\hline
                X_{(8,7)}(\gamma_{8})\cdot&X_{(7,6)}(\gamma_{7})\cdot&X_{(6,4)}(\gamma_{6})\cdot&X_{(4,5)}(0)\cdot\\\hline
                X_{(9,8)}(\gamma_{4})\cdot&X_{(8,7)}(\gamma_{3})\cdot&X_{(7,6)}(\gamma_{2})\cdot&\textcolor{red}{X_{(6,5)}(\gamma_{1})}\\\hline
            \end{array}
    \end{equation*}
        Now we conjugate every entry between the red cells by $W_{6,5}(\gamma_1)$ to get the following:
        \begin{equation*}
        \Tilde{R}_D = \begin{array}{|c|c|c|c|}
        \hline
               X_{(5,4)}(\gamma_{20})\cdot  & X_{(5,3)}(\gamma_{19})\cdot&X_{(3,2)}(\gamma_{18})\cdot& X_{(2,1)}(\gamma_{17})\cdot \\\hline
               X_{(6,4)}(\gamma_{16})\cdot  & \textcolor{red}{X_{(6,5)}(\gamma_{15}+\gamma_1)X_{5,6}(-\frac{1}{\gamma_1})X_{6,5}(\gamma_1)}\cdot&X_{(6,3)}(-\gamma_{14}\gamma_1)\cdot&X_{(6,2)}(-\gamma_{13}\gamma_1)\cdot\\\hline
                X_{(7,4)}(\gamma_{12})\cdot&X_{(4,5)}(0)\cdot&X_{(4,3)}(\gamma_{10})\cdot&X_{(3,6)}(0)\cdot\\\hline
                X_{(8,7)}(\gamma_{8})\cdot&X_{(7,5)}(\gamma_{7}\gamma_1)\cdot&X_{(5,4)}(\frac{\gamma_{6}}{\gamma_1})\cdot&X_{(4,6)}(0)\cdot\\\hline
                X_{(9,8)}(\gamma_{4})\cdot&X_{(8,7)}(\gamma_{3})\cdot&X_{(7,5)}(\gamma_{2}\gamma_1)\cdot&\textcolor{red}{X_{6,5}(-\gamma_1)X_{(5,6)}(\frac{1}{\gamma_{1}})}\\\hline
            \end{array}
    \end{equation*}
    Now, at each step, we will move the leftmost excited factors. We colour the excited factors as red, and the cooling site of the leftmost excited factor will be coloured in blue.
    \begin{align*}
    \Tilde{R}_D &= \begin{array}{|c|c|c|c|}
    \hline
               X_{(5,4)}(\gamma_{20})\cdot  & X_{(5,3)}(\gamma_{19})\cdot&X_{(3,2)}(\gamma_{18})\cdot& X_{(2,1)}(\gamma_{17})\cdot \\
               \hline
               X_{(6,4)}(\gamma_{16})\cdot  & X_{(6,5)}(\gamma_{15}+\gamma_1)X_{5,6}(-\frac{1}{\gamma_1})\textcolor{blue}{X_{6,5}(\gamma_1)}\cdot&X_{(6,3)}(-\gamma_{14}\gamma_1)\cdot&X_{(6,2)}(-\gamma_{13}\gamma_1)\cdot\\\hline
                X_{(7,4)}(\gamma_{12})\cdot&X_{(4,5)}(0)\cdot&X_{(4,3)}(\gamma_{10})\cdot&X_{(3,6)}(0)\cdot\\\hline
                X_{(8,7)}(\gamma_{8})\cdot&X_{(7,5)}(\gamma_{7}\gamma_1)\cdot&X_{(5,4)}(\frac{\gamma_{6}}{\gamma_1})\cdot&X_{(4,6)}(0)\cdot\\\hline
                X_{(9,8)}(\gamma_{4})\cdot&X_{(8,7)}(\gamma_{3})\cdot&X_{(7,5)}(\gamma_{2}\gamma_1)\cdot&\textcolor{red}{X_{6,5}(-\gamma_1)}X_{(5,6)}(\frac{1}{\gamma_{1}})\\\hline
            \end{array}\\
             &= \begin{array}{|c|c|c|c|}
             \hline
               X_{(5,4)}(\gamma_{20})\cdot  & X_{(5,3)}(\gamma_{19})\cdot&X_{(3,2)}(\gamma_{18})\cdot& X_{(2,1)}(\gamma_{17})\cdot \\\hline
               \textcolor{blue}{X_{(6,4)}(\gamma_{16})}\cdot  & X_{(6,5)}(\gamma_{15}+\gamma_1)X_{5,6}(-\frac{1}{\gamma_1})\cdot&X_{(6,3)}(-\gamma_{14}\gamma_1)\cdot&X_{(6,2)}(-\gamma_{13}\gamma_1)\cdot\\\hline
                X_{(7,4)}(\gamma_{12})\cdot&X_{(4,5)}(0)\cdot&X_{(4,3)}(\gamma_{10})\cdot&X_{(3,6)}(0)\cdot\\\hline
                X_{(8,7)}(\gamma_{8})\cdot&X_{(7,5)}(\gamma_{7}\gamma_1)\cdot&X_{(5,4)}(\frac{\gamma_{6}}{\gamma_1})\textcolor{red}{X_{(6,4)}(\gamma_6)}\cdot&X_{(4,6)}(0)\cdot\\\hline
                X_{(9,8)}(\gamma_{4})\cdot&X_{(8,7)}(\gamma_{3})\cdot&X_{(7,5)}(\gamma_{2}\gamma_1)\cdot&X_{(5,6)}(\frac{1}{\gamma_{1}})\\\hline
            \end{array}\\
            &= \begin{array}{|c|c|c|c|}
            \hline
               \textcolor{blue}{X_{(5,4)}(\gamma_{20})}\cdot  & X_{(5,3)}(\gamma_{19})\cdot&X_{(3,2)}(\gamma_{18})\cdot& X_{(2,1)}(\gamma_{17})\cdot \\\hline
               X_{(6,4)}(\gamma_{16}+\gamma_6)\cdot  & X_{(6,5)}(\gamma_{15}+\gamma_1)X_{5,6}(-\frac{1}{\gamma_1})\textcolor{red}{X_{(5,4)}(\frac{-\gamma_6}{\gamma_1})}\cdot&X_{(6,3)}(-\gamma_{14}\gamma_1)\cdot&X_{(6,2)}(-\gamma_{13}\gamma_1)\cdot\\\hline
                X_{(7,4)}(\gamma_{12})\cdot&X_{(4,5)}(0)\cdot&X_{(4,3)}(\gamma_{10})\textcolor{red}{X_{(6,3)}(-\gamma_6\gamma_{10})}\cdot&X_{(3,6)}(0)\cdot\\\hline
                X_{(8,7)}(\gamma_{8})\cdot&X_{(7,5)}(\gamma_{7}\gamma_1)\cdot&X_{(5,4)}(\frac{\gamma_{6}}{\gamma_1})\cdot&X_{(4,6)}(0)\cdot\\\hline
                X_{(9,8)}(\gamma_{4})\cdot&X_{(8,7)}(\gamma_{3})\cdot&X_{(7,5)}(\gamma_{2}\gamma_1)\cdot&X_{(5,6)}(\frac{1}{\gamma_{1}})\\\hline
            \end{array}\\
            &= \begin{array}{|c|c|c|c|}
            \hline
               X_{(5,4)}(\gamma_{20}-\frac{\gamma_6}{\gamma_1})\cdot  & X_{(5,3)}(\gamma_{19})\cdot&X_{(3,2)}(\gamma_{18})\cdot& X_{(2,1)}(\gamma_{17})\cdot \\\hline
               \textcolor{blue}{X_{(6,4)}(\gamma_{16}+\gamma_6)}\cdot  & X_{(6,5)}(\gamma_{15}+\gamma_1)\textcolor{red}{X_{(6,4)}(\frac{-\gamma_6\gamma_{15}}{\gamma_1}-\gamma_6)}X_{5,6}(-\frac{1}{\gamma_1})\cdot&X_{(6,3)}(-\gamma_{14}\gamma_1)\cdot&X_{(6,2)}(-\gamma_{13}\gamma_1)\cdot\\\hline
                X_{(7,4)}(\gamma_{12})\cdot&X_{(4,5)}(0)\cdot&X_{(4,3)}(\gamma_{10})\textcolor{red}{X_{(6,3)}(-\gamma_6\gamma_{10})}\cdot&X_{(3,6)}(0)\cdot\\\hline
                X_{(8,7)}(\gamma_{8})\cdot&X_{(7,5)}(\gamma_{7}\gamma_1)\cdot&X_{(5,4)}(\frac{\gamma_{6}}{\gamma_1})\cdot&X_{(4,6)}(0)\cdot\\\hline
                X_{(9,8)}(\gamma_{4})\cdot&X_{(8,7)}(\gamma_{3})\cdot&X_{(7,5)}(\gamma_{2}\gamma_1)\cdot&X_{(5,6)}(\frac{1}{\gamma_{1}})\\\hline
            \end{array}\\
            &= \begin{array}{|c|c|c|c|}
            \hline
               X_{(5,4)}(\gamma_{20}-\frac{\gamma_6}{\gamma_1})\cdot  & X_{(5,3)}(\gamma_{19})\cdot&X_{(3,2)}(\gamma_{18})\cdot& X_{(2,1)}(\gamma_{17})\cdot \\\hline
               X_{(6,4)}(\gamma_{16}-\frac{\gamma_6\gamma_{15}}{\gamma_1})\cdot  & X_{(6,5)}(\gamma_{15}+\gamma_1)X_{5,6}(-\frac{1}{\gamma_1})\cdot&\textcolor{blue}{X_{(6,3)}(-\gamma_{14}\gamma_1)}\cdot&X_{(6,2)}(-\gamma_{13}\gamma_1)\cdot\\\hline
                X_{(7,4)}(\gamma_{12})\cdot&X_{(4,5)}(0)\cdot&X_{(4,3)}(\gamma_{10})\textcolor{red}{X_{(6,3)}(-\gamma_6\gamma_{10})}\cdot&X_{(3,6)}(0)\cdot\\\hline
                X_{(8,7)}(\gamma_{8})\cdot&X_{(7,5)}(\gamma_{7}\gamma_1)\cdot&X_{(5,4)}(\frac{\gamma_{6}}{\gamma_1})\cdot&X_{(4,6)}(0)\cdot\\\hline
                X_{(9,8)}(\gamma_{4})\cdot&X_{(8,7)}(\gamma_{3})\cdot&X_{(7,5)}(\gamma_{2}\gamma_1)\cdot&X_{(5,6)}(\frac{1}{\gamma_{1}})\\\hline
            \end{array}\\
              &=  \begin{array}{|c|c|c|c|}
              \hline
               X_{(5,4)}(\gamma_{20}-\frac{\gamma_6}{\gamma_1})\cdot  & X_{(5,3)}(\gamma_{19})\cdot&X_{(3,2)}(\gamma_{18})\cdot& X_{(2,1)}(\gamma_{17})\cdot \\\hline
               X_{(6,4)}(\gamma_{16}-\frac{\gamma_6\gamma_{15}}{\gamma_1})\cdot  & X_{(6,5)}(\gamma_{15}+\gamma_1)X_{5,6}(-\frac{1}{\gamma_1})\cdot&X_{(6,3)}(-\gamma_{14}\gamma_1-\gamma_6\gamma_{10})\cdot&X_{(6,2)}(-\gamma_{13}\gamma_1)\cdot\\\hline
                X_{(7,4)}(\gamma_{12})\cdot&X_{(4,5)}(0)\cdot&X_{(4,3)}(\gamma_{10})\cdot&X_{(3,6)}(0)\cdot\\\hline
                X_{(8,7)}(\gamma_{8})\cdot&X_{(7,5)}(\gamma_{7}\gamma_1)\cdot&X_{(5,4)}(\frac{\gamma_{6}}{\gamma_1})\cdot&X_{(4,6)}(0)\cdot\\\hline
                X_{(9,8)}(\gamma_{4})\cdot&X_{(8,7)}(\gamma_{3})\cdot&X_{(7,5)}(\gamma_{2}\gamma_1)\cdot&X_{(5,6)}(\frac{1}{\gamma_{1}})\\\hline
            \end{array}
    \end{align*}
    The factors arranged in cells like above is the $D'$-distortion of $\Tilde{R}_{D}$. 
    \end{example}
    Note that all cooling sites are cells where $D'$ has black stones. Furthermore, when an excited factor passes through a cell, it doesn't change the entries of the non-excited factors in the cell. These observations give us the following Lemma:
    \begin{lemma}\label{lemma:newcellparameters}
        Let $D,D',c,c'$ and $i$ be as in \cref{theorem:closurerelation}. Let $\{\gamma_b\}$ be the set of parameters realizing $\Tilde{R}_D$. Let $A_b$ be the factors at cell $b$ in $D'$ distortion of $\Tilde{R}_D$. Then:
        \begin{equation*}
            A_b = \begin{cases}
                X_{i+1,i}(\gamma_c+\gamma_{c'})X_{i,i+1}(-\frac{1}{\gamma_c}) & \texttt{if\;}b=c'\\
                X_{i,i+1}(\frac{1}{\gamma_c})&\texttt{if\;} b=c\\
                X_{\sigma_{D'}(b)}(\gamma'_b)& \texttt{if\;}b\succ c \texttt{\;and\;}D'(b)=+\\
                X_{\sigma_{D'}(b)}(\gamma'_b+C)& \texttt{if\;}b\succ c \texttt{\;and\;}D'(b)=\bullet\\
                X_{\sigma_{D'}(b)}(\gamma_b)&\texttt{otherwise}
            \end{cases}
        \end{equation*}
        where $C$ is the sum of all entries in the excited factors which had $b$ as their cooling sites, and if $\sigma_{D'}(b) = (k,l)$ then
        \begin{equation*}
            \gamma_b'=\begin{cases}
                \gamma_b&\texttt{if }\{k,l\}\cap \{i,i+1\} = \emptyset\\
                -\gamma_b\gamma_c&\texttt{if } k=i,l<i\\
                \gamma_b\gamma_c &\texttt{if } l=i+1,k>i+1\\
                -\frac{\gamma_b}{\gamma_c} &\texttt{if } l=i, k>i+1\\
                \frac{\gamma_b}{\gamma_c}&\texttt{if }k=i+1,l<i\\
        \end{cases}.
        \end{equation*}
        
    \end{lemma}
    \begin{proof}[Proof of \cref{theorem:closurerelation}]
        Let $R_{D'}$ be a point in $\mathcal{D}_D'$, realized in the restricted path parametrization, and let $\beta_b$ be the parameters at each cell which realize $R_{D'}$. Using the parameters, we construct $\Tilde{R}_{D'}$. We also construct $D'$-distortion of $\Tilde{R}_D$. Now we use \cref{lemma:newcellparameters} to get the following set of equations at each cell $b$:
        \begin{align*}
            \gamma_c+\gamma_{c'}&= \beta_{c'}& \mathrm{\;if\;}b=c'\\
                \gamma'_b&=\beta_b& \mathrm{if\;}b\succ c \mathrm{\;and\;}D'(b)=+\\
                \gamma'_b+C&=\beta_b& \mathrm{if\;}b\succ c \mathrm{\;and\;}D'(b)=\bullet\\
                \gamma_b&=\beta_b&\mathrm{otherwise}
        \end{align*}
        Using the above equations we solve for $\gamma_b$ where $b\not =c$ in terms of $\{\beta_{b'}\}_{b'\in\lambda}$ and $\gamma_c$ in the reading order. We will always be able to solve this because the above system of equations is triangular ($C$ will have parameters showing up from cells before $b$ in our chosen reading order). Note from the definition of $\gamma_b'$, we get that if $D(b)= D'(b) = +$, then $\gamma_b\not=0\iff\gamma_c\not=0$, so if we truncate $\Tilde{R}_D$ to get $R_D$, we get $R_D\in \mathcal{D}_D$ for all values of $\gamma_c\not=0$. Moreover, the way we have set up the equations, it is obvious that $\lim_{\gamma_c\to \infty}\Tilde{R}_D = \Tilde{R_{D'}}$. Since limits behave well with respect to truncation, we get $\lim_{\gamma_c\to\infty}R_D = R_{D'}$. Since $R_{D'}$ was an arbitrary point in $\mathcal{D}_{D'}$ and we obtained a sequence of point in $\mathcal{D}_D$ converging to $R_{D'}$, we get the required closure relation.
    \end{proof}
    \begin{example} \label{ex:closure}
        We continue \cref{ex:distort}. Let $R_{D'}$ be an arbitrary point in $\mathcal{D}_{D'}$ realized in the restricted path parametrization, and let $\beta_b$ be the set of parameters that realize $R_{D'}$. Now using \cref{prodlemma}, with the reading order in \cref{ex:distort}, we get:
        \begin{equation*}
    \Tilde{R}_{D'} = \begin{array}{|c|c|c|c|}\hline
               X_{(5,4)}(\beta_{20})\cdot  & X_{(5,3)}(\beta_{19})\cdot&X_{(3,2)}(\beta_{18})\cdot& X_{(2,1)}(\beta_{17})\cdot \\\hline
               X_{(6,4)}(\beta_{16})\cdot  & X_{(6,5)}(\beta_{15})\cdot&X_{(6,3)}(\beta_{14})\cdot&X_{(6,2)}(\beta_{13})\cdot\\\hline
                X_{(7,4)}(\beta_{12})\cdot&X_{(4,5)}(0)\cdot&X_{(4,3)}(\beta_{10})\cdot&X_{(3,6)}(0)\cdot\\\hline
                X_{(8,7)}(\beta_{8})\cdot&X_{(7,5)}(\beta_{7})\cdot&X_{(5,4)}(\beta_6)\cdot&X_{(4,6)}(0)\cdot\\\hline
                X_{(9,8)}(\beta_{4})\cdot&X_{(8,7)}(\beta_{3})\cdot&X_{(7,5)}(\beta_2)\cdot&X_{(5,6)}(0)\\\hline
            \end{array}
    \end{equation*}
    The $D'$ distortion of $\Tilde{R}_D$ is:
    \begin{equation*}
        \Tilde{R}_D = \begin{array}{|c|c|c|c|}
              \hline
               X_{(5,4)}(\gamma_{20}-\frac{\gamma_6}{\gamma_1})\cdot  & X_{(5,3)}(\gamma_{19})\cdot&X_{(3,2)}(\gamma_{18})\cdot& X_{(2,1)}(\gamma_{17})\cdot \\\hline
               X_{(6,4)}(\gamma_{16}-\frac{\gamma_6\gamma_{15}}{\gamma_1})\cdot  & X_{(6,5)}(\gamma_{15}+\gamma_1)X_{5,6}(-\frac{1}{\gamma_1})\cdot&X_{(6,3)}(-\gamma_{14}\gamma_1-\gamma_6\gamma_{10})\cdot&X_{(6,2)}(-\gamma_{13}\gamma_1)\cdot\\\hline
                X_{(7,4)}(\gamma_{12})\cdot&X_{(4,5)}(0)\cdot&X_{(4,3)}(\gamma_{10})\cdot&X_{(3,6)}(0)\cdot\\\hline
                X_{(8,7)}(\gamma_{8})\cdot&X_{(7,5)}(\gamma_{7}\gamma_1)\cdot&X_{(5,4)}(\frac{\gamma_{6}}{\gamma_1})\cdot&X_{(4,6)}(0)\cdot\\\hline
                X_{(9,8)}(\gamma_{4})\cdot&X_{(8,7)}(\gamma_{3})\cdot&X_{(7,5)}(\gamma_{2}\gamma_1)\cdot&X_{(5,6)}(\frac{1}{\gamma_{1}})\\\hline
            \end{array}
     \end{equation*}
     Now equating the entries that show up in our factors at each cell, we get the following set of equations at each cell:
    \begin{align*}
        \gamma_{20}-\frac{\gamma_{6}}{\gamma_1}&=\beta_{20} & \gamma_{19}&=\beta_{19} & \gamma_{18}&=\beta_{18} & \gamma_{17}&=\beta_{17}\\
        \gamma_{16}-\frac{\gamma_6\gamma_{15}}{\gamma_1} &= \beta_{16}& \gamma_{15}+\gamma_1 &=\beta_{15}&-\gamma_{14}\gamma_1- \gamma_6\gamma_{10}&=\beta_{14}& -\gamma_{13}\gamma_1&=\beta_{13}\\
        \gamma_{12}&=\beta_{12}&&&\gamma_{10}&=\beta_{10}&&\\
        \gamma_8&=\beta_{8}& \gamma_{7}\gamma_1&=\beta_7&\frac{\gamma_6}{\gamma_1}&=\beta_6&&\\
        \gamma_4&=\beta_4&\gamma_3&=\beta_3&\gamma_2\gamma_1&=\beta_2&&
    \end{align*}
    If we solve for $\gamma_i, i\not=1$ in terms of $\beta_b$ and $\gamma_1$, we get $\lim_{\gamma_1\to \infty}\Tilde{R}_D = \Tilde{R}_{D'}$. Finally truncating the matrices, we get $\lim_{\gamma_1\to\infty}R_D = R_{D'}$. Since $R_{D'}$ was an arbitrary point of $\mathcal{D}_{D'}$, we get $\mathcal{D}_{D'}\subset \overline{\mathcal{D}_D}$.
    \end{example}
    \subsection{Generalizing \cref{theorem:closurerelation}} \label{sec:generalclosure}
    In this section, we will look at various generalizations of \cref{theorem:closurerelation}. We can drop the requirement $\pi(D') = 1$, by looking at the components in the Talaska-Williams parametrization. We also give a conjecture that allows us to remove the restriction on the labels of the pipes. We start by showing some lemmas:
    \begin{lemma}\label{trunc}
        Let D, D' be two Go-diagrams inside the same partition $\lambda$, and let $D_{\textup{\textbf{trunc}}}$ and $D'_{\textup{\textbf{trunc}}}$ be Go-diagrams obtained by removing the leftmost column (or topmost row) from $D$ and $D'$ respectively. If we have $\mathcal{D}_{D'}\subset \overline{\mathcal{D}_D}$ then $\mathcal{D}_{D'_{\textup{\textbf{trunc}}}}\subset \overline{\mathcal{D}_{D_{\textup{\textbf{trunc}}}}}$.
    \end{lemma}
    \begin{proof}
        We prove the statement for the case where the topmost row is removed; the statement about the columns is obtained by examining the dual of the Grassmannian. Given a point $V \in \mathcal{D}_{D'_{\textup{\textbf{trunc}}}}$, we can identify the edge weights in the Talaska William network so that the weight matrix of that network realizes $V$. Let $W_V$ be the representative. Now we can generate a matrix representative $M_V$ of an element in $\mathcal{D}(D')$ from $W_V$ by arbitrarily deciding the weights for the first row, and keeping the others the same as those for $W_V$. Therefore, $W_V$ is a submatrix of $M_V$. Now restricting ourselves to the coordinate chart $\Delta_{I_\lambda} = 1$ in $Gr_{k,n}$, we get a sequence of matrices in $\mathcal{D}_D$ which converge pointwise to $M_V$. The corresponding sequence of submatrices lies in $\mathcal{D}_{D_{\textup{\textbf{trunc}}}}$ and they converge to $W_V$.
    \end{proof}
    Given a Go-diagram, we can make a bigger Go-diagram by padding that Go-diagram with a row on top (or column on left) and filling the new cells with either $+$ or $\bullet$, while making sure the forbidden configuration doesn't arise.
    \begin{lemma}\label{pad}
        Let $D$ be a Go-diagram in $\lambda$. There exists a sequence of padding operations where we pad $D$ with columns on the left and rows on the top, such that the resultant Go-diagram is a subexpression of identity. Moreover, this padding sequence is entirely determined by $\pi(D)$.
    \end{lemma}
    \begin{proof}
        If $D$ is already not a subexpression of identity, then let $v = \pi(D)$. We can find the smallest $i$ such that $v(i) > v(i + 1)$. Now, if $v(i)$ is on the west boundary (resp. north boundary) and not the topmost edge (resp. leftmost edge), then padding $D$ on the left (resp. top) will force us to introduce a black stone. If $v(i)$ is on the boundary of the top-left corner, then we pad $D$ with both a row and a column, which will force a black stone on the new corner. In all cases, the length of the resultant permutation decreases, so we can use induction on the length of $v$ to finish the proof.
    \end{proof}
    Now using the above Lemmas, \cref{theorem:closurerelation}, and noting that the padding sequence in Lemma \ref{pad} is only dependent on the permutation of the Go-diagram, we get the following:
    \begin{theorem}\label{thm:closemain}
        Let $D'$ be a Go-diagram in some partition $\lambda$, and let $c\prec c'$ be a crossing and uncrossing pair in $D'$ with $D'(c) = \circ$. Suppose the pipes passing through $c$ are labelled $i$ and $i+1$ for some $1\leq i<n$.
        Let $D$ be the diagram obtained by replacing the stones at $c$ and $c'$ with $+$.
        Then $D$ is a Go-diagram and $\mathcal{D}_D' \subset \overline{\mathcal{D}_D}$.
    \end{theorem}
    \begin{figure}[h]
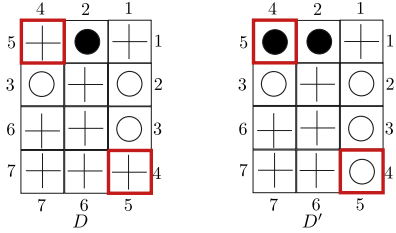

            \centering
            \includesvg[width=30mm, height=30mm]{paddconsD}
            \hspace{2em}
            \includesvg[width=30mm, height=30mm]{paddconsD_2}
            \caption{Go-diagrams $D$ and $D'$}
            \label{fig:paddex}
        \end{figure}
    \begin{figure}[h]
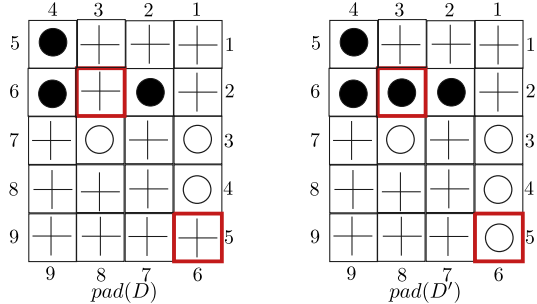

            \centering
            \includesvg[width=30mm, height=40mm]{paddedconsD}
            \hspace{2em}
            \includesvg[width=30mm, height=40mm]{paddedconsD_2}
            \caption{Go-diagrams $pad(D)$ and $pad(D')$}
            \label{fig:paddedex}
        \end{figure}
    We give an example to illustrate the ideas above:
    \begin{example}
        Let the Go-diagrams $D$ and $D'$ be the ones in \cref{fig:paddex}. Here, the cells with a crossing-uncrossing pair are highlighted in red. We use $pad(D), pad(D')$ to denote the padding obtained in \cref{pad}, these are the ones in \cref{fig:paddedex}. Note they have the same shape as $\pi(D)=\pi(D')$. Using \cref{theorem:closurerelation} we have $\mathcal{D}_{pad(D')}\subset \overline{\mathcal{D}_{pad(D)}}$. Now using \cref{trunc} we get $\mathcal{D}_{D'}\subset \overline{\mathcal{D}_D}$
    \end{example}
    \begin{remark}\label{rem:closing}
        \cref{theorem:closurerelation} is a special case of the conjecture made by Marcott \cite{marcott}. His conjecture deals with the case where the pipes at the crossing and uncrossing pair are not necessarily adjacent. In this case, when we replace the stones in cells $c$ and $c'$ with a $+$ to get $D$, we may not necessarily get a Go-diagram. Marcott deals with this case by introducing the notion of \textbf{corrective flips} to convert a diagram which is not a Go-diagram into a Go-diagram. Even though his conjecture is not true in general, we believe that it can be fixed by adding additional constraints. So we give a conjecture in the special setting where we don't need corrective flips to fix $D$:
        \begin{conjecture}\label{conjecture1}
         Let $D'$ be a Go-diagram in some partition $\lambda$, and let $c\prec c'$ be a crossing and uncrossing pair in $D'$ with $D'(c) = \circ$.
         Let $D$ be the diagram obtained by replacing the stones at $c$ and $c'$ with $+$. Let $i<j$ be labels of the pipes crossing at $c$ in $D'$.
        \begin{enumerate}
        \setlength\itemsep{0.01em}
            \item $D$ is a Go-diagram after recolouring stones appropriately, i.e., $D$ does not have any cells with the forbidden configuration.
            \item $i$ and $j$ don't form a crossing and uncrossing pair with any $k$, $i<k<j$ between $c$ and $c'$ in $D'$.
        \end{enumerate}
        
        Then $\mathcal{D}_{D'} \subset \overline{\mathcal{D}_D}$.
    \end{conjecture}
     In the case where we don't need to fix $D$ using corrective flips, Marcott's conjecture is the same as above, except he doesn't have the second condition. Even though we know that the first condition is not sufficient, it might be possible to replace the second condition with a weaker condition. We also conjecture that the second condition is necessary as well.
        \begin{conjecture}\label{conjecture2}
         Let $D'$ be a Go-diagram in some partition $\lambda$, and let $c\prec c'$ be a crossing and uncrossing pair in $D'$ with $D'(c) = \circ$.
         Let $D$ be the diagram obtained by replacing the stones at $c$ and $c'$ with $+$. Let $i<j$ be labels of the pipes crossing at $c$ in $D'$.
        \begin{enumerate}
        \setlength\itemsep{0.01em}
            \item $D$ is a Go-diagram after recolouring stones appropriately, i.e., $D$ does not have any cells with the forbidden configuration.
            \item There is $k$ with $i<k<j$ such that $k$ forms a crossing-uncrossing pair with $i$ or $j$ between $c$ and $c'$ in $D'$.
        \end{enumerate}
        
        Then $\mathcal{D}_{D'} \not\subset \overline{\mathcal{D}_D}$.
    \end{conjecture}
    For the Go-diagrams $D'$ in \cref{conjecture2}, even though we don't have a closure relation, the intersection $\mathcal{D}_{D'} \cap \overline{\mathcal{D}_D}$ can be non-empty.
    \end{remark}

\section*{Acknowledgments}
    We want to thank Kevin Purbhoo and Olya Mandelshtam for their constant guidance and support throughout this project. We want to thank Cameron Marcott for his thesis, which provided the background for this work. Finally, we thank Oliver Pechenik, Allen Knutson and David Speyer for helpful discussions.
    
\printbibliography
\end{document}